\documentclass[12pt]{article}
\usepackage{fullpage}
\usepackage{graphicx}
\usepackage{amssymb,amsmath,amsthm}
\usepackage{epstopdf}
\usepackage{color}

\newtheorem{theorem}{Theorem}[section]
\newtheorem{corollary}[theorem]{Corollary}

\newtheorem{conjecture}[theorem]{Conjecture}

\newtheorem{lemma}[theorem]{Lemma}
\theoremstyle{definition}
\newtheorem{definition}[theorem]{Definition}

\def\qed{\hfill\ifhmode\unskip\nobreak\fi\quad\ifmmode\Box\else\hfill$\Box$\fi\\}

\long\def\skipit#1{}
\def\VEC#1#2#3{#1_{#2},\ldots,#1_{#3}}

\def\SE#1#2#3{\sum_{#1=#2}^{#3}}

\def\CH#1#2{\binom{#1}{#2}}
\def\FR#1#2{\frac{#1}{#2}}
\def\FL#1{\left\lfloor{#1}\right\rfloor} 
\def\CL#1{\left\lceil{#1}\right\rceil}   
\def\C#1{\left\vert #1\right\vert}

\def\la{\langle}
\def\ra{\rangle}
\def\esub{\subseteq}
\def\Sp#1{S^{#1}}
\def\nosub{\not\subseteq}

\def\NN{{\mathbb N}} 
\def\cC{{\mathcal C}}
\def\cD{{\mathcal D}}
\def\cF{{\mathcal F}}

\def\hC{{\hat C}}
\def\hL{{\hat L}}

\def\hq{{\hat q}}

\def\hR{{\hat R}}
\def\hX{{\hat X}}
\def\hx{{\hat x}}
\def\c{11}

\begin{document}

\title{Trees with at least $6\ell+\c$ vertices
are $\ell$-reconstructible}

\author{
Alexandr V. Kostochka\thanks{University of Illinois,
Urbana IL, and Sobolev Institute of Mathematics, Novosibirsk,
Russia: \texttt{kostochk@math.uiuc.edu}.  Supported by NSF
grant DMS-2153507 and NSF RTG grant DMS-1937241.}\,,
Mina Nahvi\thanks{University of Illinois,
Urbana IL: \texttt{mnahvi2@illinois.edu}.}\,,
Douglas B. West\thanks{Zhejiang Normal Univ., Jinhua, China
and Univ.\ of Illinois, Urbana IL:
\texttt{dwest@illinois.edu}.  Supported by National Natural Science Foundation
of China grants NSFC 11871439, 11971439, and U20A2068.}\,,
Dara Zirlin\thanks{University of Illinois at Urbana--Champaign, Urbana IL 61801:
\texttt{zirlin2@illinois.edu}.  Supported in part by Arnold O. Beckman Campus
Research Board Award RB20003 of the University of Illinois.}
}

\date{\today}
\maketitle

\baselineskip 16pt

\begin{abstract}
The {\it $(n-\ell)$-deck} of an $n$-vertex graph is the multiset of (unlabeled)
subgraphs obtained from it by deleting $\ell$ vertices.  An $n$-vertex graph is
{\it $\ell$-reconstructible} if it is determined by its $(n-\ell)$-deck,
meaning that no other graph has the same deck.  We prove that every tree with
at least $6\ell+\c$ vertices is $\ell$-reconstructible.
\end{abstract}

\section{Introduction}

The {\it $j$-deck} of a graph is the multiset of its $j$-vertex induced
subgraphs.  We write this as the $(n-\ell)$-deck when the graph has $n$
vertices and the focus is on deleting $\ell$ vertices.  An $n$-vertex graph is
{\it $\ell$-reconstructible} if it is determined by its $(n-\ell)$-deck.
Since every member of the $(j-1)$-deck arises $n-j+1$ times by deleting a
vertex from a member of the $j$-deck, the $j$-deck of a graph determines its
$(j-1)$-deck.  Therefore, a natural reconstruction problem is to find for each
graph the maximum $\ell$ such that it is $\ell$-reconstructible.  For this
problem, Manvel~\cite{M69,M74} extended the classical Reconstruction
Conjecture of Kelly~\cite{Kel1} and Ulam~\cite{U}.

\begin{conjecture}[{\rm Manvel~\cite{M69,M74}}]\label{manvel}
For $l\in\NN$, there exists a threshold $M_\ell$ such that every graph with
at least $M_\ell$ vertices is $\ell$-reconstructible.
\end{conjecture}

\noindent
Manvel named this ``Kelly's Conjecture'' in honor of the final sentence in
Kelly~\cite{Kel2}, which suggested that one can study reconstruction from
the $(n-2)$-deck.  Manvel noted that Kelly may have expected the statement to
be false.

The classical Reconstruction Conjecture is $M_1=3$.  Lacking a proof of
Conjecture~\ref{manvel} for any fixed $\ell$, we study threshold numbers of
vertices for $\ell$-reconstructibility of graphs in special classes.  The
survey by Kostochka and West~\cite{KW} describes prior such results.  Here our
aim is to reduce the threshold number of vertices to guarantee
$\ell$-reconstructibility of trees.

Reconstruction arguments for special families have two parts, named (when
$\ell=1$) by Bondy and Hemminger~\cite{BH}.  When the $(n-\ell)$-deck
guarantees that all reconstructions or no reconstructions lie in the specified
family, the family is {\it $\ell$-recognizable}.  Separately, using the
knowledge that every reconstruction from the deck is in the family, one
determines that only one graph in the family has that deck; this makes the
family {\it weakly $\ell$-reconstructible}.  Together, the two steps make
graphs in the family $\ell$-reconstructible.
 
N\'ydl~\cite{N90} conjectured that trees with at least $2\ell+1$ vertices are
weakly $\ell$-reconstructible, having presented in~\cite{N81} two trees with
$2\ell$ vertices having the same $\ell$-deck, to make the conjecture sharp.
The two trees arise from a path with $2\ell-1$ vertices by adding one leaf,
adjacent either to the central vertex of the path or to one of its neighbors.
Kostochka and West~\cite{KW} used the results of Spinoza and West~\cite{SW} to
give a short proof of N\'ydl's result.

When $\ell=2$, the disjoint union of a $4$-cycle and an isolated vertex has the
same deck as these two trees, so $5$-vertex trees are not a $2$-recognizable
family.  However, Kostochka, Nahvi, West, and Zirlin~\cite{KNWZa} proved that
$n$-vertex acyclic graphs are $\ell$-recognizable when $n\ge2\ell+1$, except
for $(n,\ell)=(5,2)$.  As noted earlier, the $(n-\ell)$-deck yields the
$k$-deck whenever $k<n-\ell$, so the $(n-\ell)$-deck also yields the $2$-deck,
which fixes the number of edges.  An $n$-vertex graph is a tree if and only if
it is acyclic and has $n-1$ edges, so the family of $n$-vertex trees is
$\ell$-recognizable when $n\ge2\ell+1$ (except for $(n,\ell)=(5,2)$).

This suggests modifying N\'ydl's conjecture to say that trees with at least
$2\ell+1$ vertices are $\ell$-reconstructible (modifying to $2\ell+2$ when
$\ell=2$).  Indeed, Kelly~\cite{Kel2} proved that trees with at least three
vertices are $1$-reconstructible, and Giles~\cite{Gil} proved that trees with
at least six vertices are $2$-reconstructible.  Hunter~\cite{Hu} proved that
caterpillars with at least $(2+o(1))\ell$ vertices are $\ell$-reconstructible.
Groenland, Johnston, Scott, and Tan~\cite{GJST} found one counterexample to
N\'ydl's conjecture for $\ell=6$: two specific trees with $13$ vertices having
the same $7$-deck.  However, they proved that a threshold number of vertices
does exist for trees:

\begin{theorem}[{\rm\cite{GJST}}]
When $j\ge{\FR89n+\FR49\sqrt{8n+5}+1}$, every $n$-vertex tree is determined by
its $j$-deck.  Thus $n$-vertex trees are $\ell$-reconstructible when
$n\ge 9\ell+24\sqrt{2\ell}+o(\sqrt\ell)$.
\end{theorem}

For $\ell=3$, this theorem applies when $n\ge194$.
Using reconstruction of rooted trees, Kostochka, Nahvi, West, and
Zirlin~\cite{KNWZ3} gave a lengthy proof of the threshold $n\ge25$ when
$\ell=3$.  Our aim in this paper is to lower the general threshold for
$\ell$-reconstructibility of trees by proving the following theorem,
which brings the threshold for $\ell=3$ down to $n\ge28$.

\begin{theorem}
When $n\ge6\ell+\c$, all $n$-vertex trees are $\ell$-reconstructible.
\end{theorem}

We close this introduction by describing the structure of the paper.
Our proof is constructive; that is, we show how to obtain the unique
$n$-vertex tree corresponding to a given $(n-\ell)$-deck $\cD$.
As noted earlier, the family of $n$-vertex trees is 
$\ell$-recognizable when $n\ge2\ell+1$~\cite{KNWZa}, so we may assume that all
reconstructions from $\cD$ are trees.

We divide the proof into various cases depending on whether the diameter
of the tree is high or low and on whether various special structures are
present in the tree.  It is important to note that we must show each case
is recognizable from the deck before we can make use of the hypotheses of
that case in reconstructing the tree.

In Section 2, we develop tools used in various cases.  We first show that
several important parameters of $T$ are $\ell$-reconstructible, including the
diameter (which we write as $r+1$), a number $k$ that is within $1$ of the
minimum radius among connected cards, and the number of vertices in $T$ from
which three edge-dispoint paths of length $\ell+1$ can be grown.  We call such
vertices ``spi-centers'', because a union of edge-disjoint paths with a common
endpoint is often called a {\it spider}.  We also introduce a counting process
called an Exclusion Argument, which is a technique used repeatedly in later
sections to obtain various subtrees in $T$.


Sections 3--5 discuss the ``high-diameter'' case, defined by $r\ge n-3\ell$.
In this case, $k\ge\ell+1$, and the cards show all connected subgraphs of $T$
with diameter ``not too large''.  Section 3 completes the case where $T$ has a
spi-center, and Sections 4--5 consider the high-diameter trees without
spi-centers.  In addition to spi-centers, we consider whether the deck has
a {\it sparse card}, which is a card containing an $r$-vertex path along
which the card has only one branch vertex.  In Section 4, we find the maximal
subtrees of diameter $2k-2$ that contain the end portions of every $r$-vertex
path in $T$.  In Section 5, we finish reconstructing $T$ by assembling these
two subtrees properly and finding the rest of the tree.



In Section 6, we consider the ``small-diameter'' case, meaning $r< n-3\ell$.
Subcases consider whether $T$ has a sparse card or not, whether the branch
vertex along the $r$-vertex path in a sparse card has degree $3$ or higher,
and whether $T$ has various other cards in which there is a long path from
a leaf to the nearest branch vertex.  Note that the existence of various kinds
of cards is immediately recognizable from the deck.


\section{Vines and Diameter}\label{tools}

In studying $\ell$-reconstructibility of $n$-vertex trees, we use different
methods for trees with large diameter and trees with small diameter.  Along
the way, we introduce various structures in trees whose presence or absence
is determined by the $(n-\ell)$-deck, and we use their occurrence or
non-occurrence in resolving various cases.

\begin{definition}\label{diam}
In a graph $G$, the {\it distance} between two vertices is the minimum length
of a path containing them.  The {\it eccentricity} of a vertex in a graph $G$
is the maximum of the distances from it to other vertices.  A {\it center} of
$G$ is a vertex of minimum eccentricity, and the minimum eccentricity is called
the {\it radius} of $G$.  The {\it diameter} of $G$ is the maximum eccentricity,
which is the maximum distance between vertices.
\end{definition}

It is an elementary exercise that a tree has one center or two adjacent centers,
when the diameter is even or odd, respectively.  The deck of an $n$-vertex tree
has no cards that are paths precisely when the tree has no paths with $n-\ell$
vertices.  We then know the diameter, because we all paths appear in the cards.
Later in this section we will determine the diameter from the deck also when it
is larger.  We will need to count subtrees with various diameters, a technique
we used in~\cite{KNWZa} in the more general situation of $n$-vertex graphs
whose $(n-\ell)$-decks have no cards containing cycles.

\begin{definition}\label{kdef}
A {\it $j$-vine} is a tree with diameter $2j$.
A {\it $j$-evine} is a tree with diameter $2j+1$.
A {\it $j$-center} or {\it $j$-central edge} is the central vertex or edge
in a $j$-vine or $j$-evine, respectively.
\end{definition}

\begin{lemma}\label{vinemax}
In a tree $T$, every $j$-vine or $j$-evine $H$ lies in a unique maximal
$j$-vine or $j$-evine, respectively.
\end{lemma}
\begin{proof}
The unique maximal such graph containing $H$ is the subgraph induced by the set
of all vertices within distance $j$ of its central vertex or central edge.
This is a $j$-vine or $j$-evine, respectively.  Any other such subgraph
containing it would have to have the same center or central edge, but then it
cannot have any additional vertices.
\end{proof}

For a family $\cF$ of graphs, an {\it $\cF$-subgraph} of a graph $G$ is an
induced subgraph of $G$ belonging to $\cF$.  For $F\in \cF$, let $m(F,G)$ be
the number of copies of $F$ that are maximal $\cF$-subgraphs in $G$.  The case
$\ell=1$ of the next lemma is due to Greenwell and Hemminger~\cite{GH}.
Similar statements for general $\ell$ appear for example in~\cite{GJST}.  We
include a proof for completeness; it is slightly simpler than proofs in the
literature involving inclusion chains.

\begin{lemma}[\cite{KNWZa}]\label{counting}
Let $\cF$ be a family of graphs such that every subgraph of $G$ belonging to
$\cF$ lies in a unique maximal subgraph of $G$ belonging to $\cF$.  If for
every $F\in\cF$ with at least $n-\ell$ vertices the value of $m(F,G)$ is known
from the $(n-\ell)$-deck of $G$, then for all $F\in \cF$ the $(n-\ell)$-deck
$\cD$ determines $m(F,G)$.
\end{lemma}
\begin{proof}
Let $t=\C{V(G)}-\C{V(F)}$; we use induction on $t$.  When $t\le\ell$, the value
$m(F,G)$ is given.  When $t>\ell$, group the induced subgraphs of $G$
isomorphic to $F$ according to the unique maximal $\cF$-subgraph of $G$
containing them (as an induced subgraph).  Counting all copies of $F$ then
yields
$$s(F,G)=\sum_{H\in{\cF}} s(F,H)m(H,G).$$
Since $\C{V(F)}<n-\ell$, we know $s(F,G)$ from the deck, and we know
$s(F,H)$ when $F$ and $H$ are known.  By the induction hypothesis, we know all
values of the form $m(H,G)$ when $F$ is an induced subgraph of $H$ except
$m(F,G)$.  Therefore, we can solve for $m(F,G)$.
\end{proof}

Next we establish consistent notation for our subsequent discussion.

\begin{definition}\label{defr}
Always $\cD$ denotes the $(n-\ell)$-deck of an $n$-vertex tree $T$;
we call $\cD$ simply the {\it deck} of $T$.  Also $r$ denotes always the
maximum number of vertices in a path in $T$.
A {\it connected subcard} or simply {\it subcard} of $T$ is a connected
subgraph of $T$ with at most $n-\ell$ vertices; it appears in a connected card.
We use ``csc'' for ``connected subcard''.
\end{definition}

\begin{definition}\label{defk}
Fix $k$ to be the largest integer $j$ such that $T$ contains a
$j$-evine and every $j$-evine in $T$ has fewer than $n-\ell$ vertices.
{\bf This fixes $k$ for the remainder of the paper.}
\end{definition}

\begin{lemma}\label{kprop}
The value of $k$ is determined by the deck of $T$.
\end{lemma}
\begin{proof}
Each edge forms a $0$-evine (which are the only $0$-evines), and every
$j$-evine with $j>0$ contains a smaller $(j-1)$-evine, so the value of $k$ is
well-defined.  It remains to compute $k$.

All subgraphs with at most $n-\ell$ vertices are visible from the deck.
The deck yields a largest value $j'$ such that some csc is a $j'$-evine and all
such cscs have fewer than $n-\ell$ vertices.  Since $k$ is an integer having
these properties, $j'\ge k$.  Since some csc is a $j'$-evine with fewer than
$n-\ell$ vertices, $2j'+2<n-\ell$.

By the choice of $j'$, no $j'$-evine in $T$ has exactly $n-\ell$ vertices,
so a smallest $j'$-evine with more than $n-\ell$ vertices can only be a
path, which would require $2j'+2>n-\ell$.  Thus no such $j'$-evine exists,
yielding $k\ge j'$.  Hence $j'=k$.
\end{proof}

\begin{lemma}\label{kvine}
In a reconstruction $T$ from $\cD$, there exist $k$-vines, and every $k$-vine
in $T$ has fewer than $n-\ell$ vertices.
\end{lemma}
\begin{proof}
By the definition of $k$, there is a $k$-evine in $T$.  A $k$-evine
contains two $k$-vines whose centers form the central edge of the $k$-evine.

If some $k$-vine has at least $n-\ell$ vertices, then let $C$ with
center $z$ be a largest $k$-vine in $T$.  Since $T$ contains a $k$-evine, the
diameter of $T$ is at least $2k+1$, and hence the radius of $T$ is at least
$k+1$.  Thus $T$ has a vertex $x$ at distance $k+1$ from $z$.  Note that $x$ is
not contained in $C$, but $x$ has a neighbor $y$ in $C$ at distance $k$ from
$z$.  Adding $x$ and the edge $xy$ to $C$ creates a $k$-evine with more than
$n-\ell$ vertices, contradicting the definition of $k$.
\end{proof}

\begin{corollary}\label{countk}
For $j\le k$, the deck $\cD$ determines the numbers of maximal $j$-evines
and $j$-vines with each isomorphism type.  All reconstructions from $\cD$ have
the same numbers of $j$-centers and $j$-central edges.
\end{corollary}
\begin{proof}
Fix $j$, and let $\cF$ be the family of $j$-vines or the family of $j$-evines.
By the definition of $k$ and Lemmas~\ref{kprop} and~\ref{kvine}, $m(F,T)=0$ for
all $F\in\cF$ having at least $n-\ell$ vertices.  Hence Lemma~\ref{counting}
applies to compute $m(F,T)$ for all $F\in\cF$.  By Lemma~\ref{vinemax}, there
is a one-to-one correspondence between the maximal $j$-vines and the
$j$-centers, and similarly for $j$-evines and $j$-central edges.
\end{proof}

Since $1$-vines are stars, setting $j=1$ in Corollary~\ref{countk} 
provides the degree list if $k\ge1$.  Groenland et al.~\cite{GJST} proved that
the degree list is $\ell$-reconstructible for all $n$-vertex graphs whenever
$n-\ell>\sqrt{2n\log(2n)}$.  Taylor~\cite{Tay} had shown that asymptotically
$n>\ell{\rm e}$ is enough.  For trees we obtain a simpler intermediate
threshold that suffices for our needs.

\begin{corollary}\label{degrees}
For $n\ge2\ell+3$, the degree list of any $n$-vertex tree is determined by
its $(n-\ell)$-deck.
\end{corollary}
\begin{proof}
All $1$-vines are stars.  Each vertex with degree at least $2$ is the center of
exactly one maximal $1$-vine.  Suppose first that no star has at least $n-\ell$
vertices.  For $t\ge3$, by the counting argument (Lemma~\ref{counting}) the
deck determines the number of maximal $1$-vines having $t$ vertices.  This is
the number of vertices with degree $t-1$ in any reconstruction.

Now suppose that some star has at least $n-\ell$ vertices.  Since $n-\ell\ge4$,
we see in the deck that there are no $3$-cycles or $4$-cycles, so
two stars share at most a common leaf or an edge joining the centers.  Having
two stars with at least $n-\ell$ vertices thus requires $n\le 2\ell+2$.

Hence only one vertex has degree at least $n-\ell-1$.  Its degree is $d$ if and
only if exactly $\CH d{n-\ell-1}$ cards are stars.  Thus we have $m(S,T)$ for
any reconstruction $T$ and every star $S$ with at least $n-\ell$ vertices.
Again Lemma~\ref{counting} applies and we obtain the number of vertices with
degree $t-1$ whenever $t\ge3$.

Since every reconstruction is a tree, the remaining vertices have degree $1$.
\end{proof}

\vspace{-1pc}
\begin{lemma}\label{diam}
Every connected card in $\cD$ has diameter at least $2k+2$,
and some connected card has diameter at most $2k+3$.
\end{lemma}
\begin{proof}
A connected card with diameter at most $2k+1$ would be a $j$-evine or $j$-vine
with $j\le k$ having $n-\ell$ vertices, contradicting the definition of $k$
or Lemma~\ref{kvine}.

For the second claim, let $C$ be a connected card.  If $C$ has diameter at
least $2k+3$, then $C$ contains a path with $2k+4$ vertices.  Hence
$n-\ell\ge 2k+4$ and $T$ contains a $(k+1)$-evine.  By the definition of $k$,
some $(k+1)$-evine has at least $n-\ell$ vertices.  Since $n-\ell\ge 2k+4$,
we can iteratively delete leaves of such a $(k+1)$-evine outside a path with
$2k+4$ vertices to trim the subtree to $n-\ell$ vertices.  We thus obtain a
card that is a $(k+1)$-evine and has diameter $2k+3$.  Hence some card has
diameter at most $2k+3$.
\end{proof}

\vspace{-.5pc}
In a tree $T$, the number $r$ of vertices in a longest path is the diameter
plus $1$.

\begin{lemma}\label{largek}
$(r-3)/2\ge k\geq(r-\ell-4)/2$.  In particular, if $r\ge 3\ell+6$,
then $k\ge\ell+1$.
\end{lemma}
\begin{proof}
Lemma~\ref{diam} guarantees a connected card with diameter at least $2k+2$.
It contains a path with at least $2k+3$ vertices, so $r\ge 2k+3$.

Let $P$ be a longest path in $T$.  By Lemma~\ref{diam}, some connected card $C$
has diameter at most $2k+3$.  Since $T$ has no cycle and $C$ has no path with
more than $2k+4$ vertices, at most $2k+4$ vertices of $P$ appear in $C$.  In
addition, $T$ has only $\ell$ vertices outside $C$.  Hence $r\le 2k+4+\ell$,
which yields the claimed lower bound on $k$.

When $r\ge3\ell+6$, the lower bound on $k$ simplifies to $k\ge\ell+1$.
\end{proof}

\vspace{-1pc}
When $r<n-\ell$, the value of $r$ is the maximum number of vertices in a path
contained in a card.  Next we show that the deck also determines $r$ when
$r\ge n-\ell$.

\begin{lemma}\label{longp}
If $n\ge 4\ell+8$ and $r\ge n-\ell$, then $k\ge\ell+2$, all paths in $T$
with more than $n-r$ vertices intersect any longest path $P$, all $k$-centers
and $(k-1)$-centers in $T$ lie on $P$, and the value of $r$ is determined by
the deck.  
\end{lemma}
\begin{proof}
With $r\ge n-\ell$ and $n\ge4\ell+8$, Lemma~\ref{largek} yields
$k\ge(n-2\ell-4)/2\ge \ell+2$.  All paths with more than $n-r$ vertices
intersect $P$, since only $n-r$ vertices exist outside $P$.

Since $\ell\ge n-r$, all paths with at least $\ell+1$ vertices intersect $P$.
Since also $k\ge\ell+2$, all paths with at least $k-1$ vertices intersect $P$.
From a vertex outside $P$, only one path leads to $P$; hence all $k$-centers
and $(k-1)$-centers lie on $P$.

Since $T$ has a $k$-vine, $r\ge 2k+1$, so $k$-centers (and $(k-1)$-centers) do
exist on $P$.  When $2j+1\le r$, all the vertices of $P$ are $j$-centers except
the $j$ vertices closest to each end.  By Corollary~\ref{countk}, the deck
determines the number $s$ of $k$-centers.  Hence $r=s+2k$.
\end{proof}

\vspace{-1pc}
When $k\ge\ell+1$, it follows from Corollary~\ref{countk} that we know the
number of $(\ell+1)$-centers in $T$.  Furthermore, we know the multiset of
maximal $(\ell+1)$-vines in $T$.  Among these, there is a particular type
of $(\ell+1)$-vine that figures heavily in our analysis.

\begin{definition}
A {\it branch vertex} in a tree is a vertex having degree at least $3$.
A {\it leg} in a non-path tree is a path from a leaf to the nearest branch
vertex.  A {\it spider} is a tree with at most one branch vertex; its
{\it degree} is the degree of the branch vertex.  The spider $S_{\VEC j1d}$
with $1+\SE i1d j_i$ vertices is the union of legs with lengths $\VEC j1d$
having a common endpoint.  A {\it spi-center} is an $(\ell+1)$-center that is
the branch vertex in a copy of $S_{\ell+1,\ell+1,\ell+1}$ in $T$.  We will
often discuss $3$-legged spiders whose three legs have the same length.  For
this special situation with legs of length $j$ we write $\Sp{j}$; that is,
$\Sp{\ell+1}=S_{\ell+1,\ell+1,\ell+1}$.
\end{definition}

\begin{corollary}\label{spi}
If $k\ge\ell+1$, then the deck determines the number of spi-centers in $T$.
\end{corollary}
\begin{proof}
When $j\le k$, by Corollary~\ref{countk} we know all the maximal $j$-vines.
The spi-centers correspond bijectively to the maximal $(\ell+1)$-vines that
contain $\Sp{\ell+1}$.
\end{proof}

\vspace{-1.5pc}
\begin{lemma}\label{spider}
Suppose $n\ge4\ell+1$ and $r<n-\ell$.  If $T$ has a card $C$ that is a
$3$-legged spider, then $T$ has no spi-center other than the branch vertex
of $C$.
\end{lemma}
\begin{proof}
A copy of $\Sp{\ell+1}$ with branch vertex other than $z$ would have a leg
completely outside $C$.  Since $T$ has only $\ell$ vertices outside $C$, this
cannot occur.
\end{proof}
%
%

\begin{definition}
In a tree $T$, an {\it offshoot} from a vertex set $S$ (or from the subgraph
induced by $S$) is a component of $T-S$, rooted at its vertex having a neighbor
in $S$.  The {\it length} of an offshoot is the maximum number of vertices in a
path in it that begins at its root.
\end{definition}

\begin{lemma}\label{central}
Let $P$ be a longest path in $T$, with vertices $\VEC v1r$.  If $v_j$ with
$j\le (r+1)/2$ is a spi-center and no vertex between $v_j$ and $v_{r+1-j}$ is a
spi-center, then all longest paths have $\VEC vj{r+1-j}$ as their $r-2j+2$
central vertices.
\end{lemma}
\begin{proof}
Since $v_j$ is a spi-center, $T$ has an offshoot from $P$ at $v_j$ with length
at least $\ell+1$.  Since $P$ is a longest path, $\ell+1\le j-1$.

All longest paths in a tree have the same central vertex (or vertex pair).
If a longest path $P'$ diverges from $P$ at some vertex $w$ between $v_j$ and
$v_{r+1-j}$, then $P'$ has at least $j$ vertices outside $P$.  Now $w$ is
a spi-center, contradicting the hypothesis.
\end{proof}

We will often consider offshoots from a longest path $P$.
The union of a vertex $z$ on $P$ together with the offshoots from $P$ at
$z$ is a rooted tree with $z$ as root.  The next lemma describes an
{\it exclusion argument} that we use in various situations to
obtain offshoots at a vertex $z$ of a tree being reconstructed.

\begin{lemma}[The Exclusion Argument]\label{count}
For a rooted tree $R$ with root $z$, the {\em $z$-offshoots} are the
components of the graph $R-z$, rooted at the neighbors of $z$.  If the largest
$z$-offshoots of $R$ are known (with multiplicities), then the complete list of
$z$-offshoots is determined by the multiset $\cal M$ of all rooted subtrees
obtained from individual $z$-offshoots.
\end{lemma}
\begin{proof}
The $z$-offshoots are determined in nonincreasing order of number of vertices.
The largest $z$-offshoots are given initially.  Having determined all
$z$-offshoots with more than $p$ vertices, let $C$ be a rooted tree with $p$
vertices in $\cal M$.  Since we know all $z$-offshoots of $R$ larger than $C$,
we know how many copies of $C$ in $\cal M$ arise from larger $z$-offshoots by
deleting some number of leaves.  The number of copies of $C$ as a $z$-offshoot
is obtained from the number of copies of $C$ in $\cal M$ by excluding those
that arise from larger $z$-offshoots.
\end{proof}

%
%

\section{Sparse Cards and $3$-Legged Spiders}

As in Section~\ref{tools}, we keep $k$ defined as in Definition~\ref{defk} and
$r$ as the maximum number of vertices in a path in the $n$-vertex tree $T$ that
we want to show is determined by its $(n-\ell)$-deck $\cD$.
In this section we show that trees containing $\Sp{\ell+1}$ (that is,
having a spi-center) are $\ell$-reconstructible when their diameter and 
number of vertices are not small.  The presence of $\Sp{\ell+1}$
guarantees $r\ge2\ell+3$, but the stronger lower bound $r\ge3\ell+6$
will guarantee $k\ge\ell+1$ (by Lemma~\ref{largek}).  The lower bound
$r\ge3\ell+6$ is implied by $n\ge6\ell+6$ and $r\ge n-3\ell$.  We first
define additional terminology for structures used in the proofs.

\begin{definition}\label{long}
An {\it $r$-path} is an $r$-vertex path; we use this term only for longest
paths in trees whose longest paths have $r$ vertices.  In the $(n-\ell)$ deck
of such a tree, a {\it long card} is a connected card containing an $r$-path.
Similarly, a {\it long csc} or {\it long subtree} is a csc or subtree
containing an $r$-path.

When $r<n-\ell$, a {\it sparse card} is a long card such that some $r$-path 
in the card contains exactly one branch vertex of the card.  That branch vertex
is the {\it primary vertex} of the card.  The {\it degree} of a sparse card $C$
is the degree in $C$ of its primary vertex.  Note that an $r$-path omits
at least one leg in any copy of $\Sp{\ell+1}$, so $\Sp{\ell+1}\esub T$ implies
$r<n-\ell$.
\end{definition}

\begin{lemma}\label{onelctr}
If $n\ge6\ell+6$ and $r\ge n-3\ell$, and $T$ contains a sparse card
and exactly one spi-center, then $T$ is $\ell$-reconstructible.
\end{lemma}
\begin{proof}
As noted before Definition~\ref{long}, the hypotheses imply $r\ge3\ell+6$,
which implies $k\ge\ell+1$.  By Lemma~\ref{kprop}, we know $k$, and then
Corollary~\ref{spi} gives us the number of spi-centers.  By Lemma~\ref{longp},
we also know $r$.  We see all cards in the deck, so we see that $T$ has a
sparse card.  Thus we recognize that $T$ has the specified properties.  As
noted in Definition~\ref{long}, $\Sp{\ell+1}\esub T$ implies $r<n-\ell$.

Let $C_1$ be a sparse card with $r$-path $P$ and primary vertex $z$.  We claim
that this vertex $z$ is the unique spi-center in $T$.  Otherwise, let
$w$ be the vertex on $P$ closest to the spi-center.  Since $P$ is a longest
path, the two portions of $P$ leaving $w$ are as long as any leg of
$\Sp{\ell+1}$ that is not on $P$, which makes $w$ a spi-center.  However, no
vertex of $P$ other than $z$ can be a spi-center, since $C_1$ being a sparse
card means that at most $\ell$ vertices lie in offshoots from $P$ at vertices
other than $z$.  Thus the unique spi-center is $z$ (see Figure~\ref{c1c2fig}).

Since we see $z$ in $C_1$, we may specify $j$ with $j\le{(r+1)/2}$ by saying
that the distance of $z$ from the closest endpoint of $P$ is $j-1$.  Since
$r\ge n-3\ell$, there are at most $2\ell$ vertices in offshoots from $P$ at
$z$.  Thus no copy of $\Sp{\ell+1}$ has two legs outside $P$.  This implies
$j-1\ge \ell+1$.

\begin{figure}[h]
\begin{center}
\includegraphics[scale=0.5]{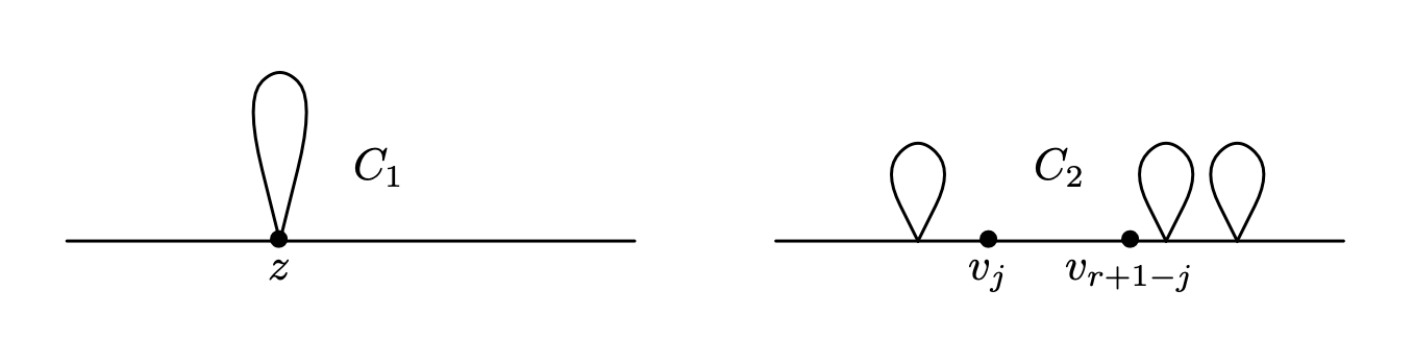}
\caption{Sparse card $C_1$ and large csc $C_2$}\label{c1c2fig}
\end{center}
\end{figure}
Let $C_2$ be a largest long csc among those having an $r$-path $P'$ with
vertices $\VEC v1r$ such that $d_{C_2}(v_j)=d_{C_2}(v_{r+1-j})=2$.  (see
Figure~\ref{c1c2fig}).  By Lemma~\ref{central}, every $r$-path in $T$ has
$\VEC vj{r+1-j}$ as its central portion, so $z\in\{v_j,v_{r+1-j}\}$.

Since $z$ is a spi-center, we omit at least $\ell+1$ vertices
in the offshoots from $P'$ at $z$, so $C_2$ has fewer than $n-\ell$ vertices.
By its maximality, $C_2$ contains in full all offshoots from $P'$ at vertices
other than $v_j$ and $v_{r+1-j}$.

Let $T_1$ and $T_2$ be the components of $T-z$ containing the neighbors of
$z$ on $P'$, with $T_1$ containing $j-1$ vertices from $P'$.  If $j<(r+1)/2$,
then let $z$ be the vertex of $\{v_j,v_{r+1-j}\}$ other than $z$; otherwise,
let $z'=z$.  Let $W$ and $W'$ be the unions of the offshoots from $P'$ in $T$
at $z$ and at $z'$, respectively.  Thus $W'\esub T_2$ if $j<(r+1)/2$, but 
$W'=W$ if $j=(r+1)/2$, in which case $T_1$ and $T_2$ are labeled arbitrarily.

\smallskip
{\bf Claim:}
{\it If $T_1$ and $T_2$ are known,
then $T$ can be reconstructed from the deck.}
Since $T_2$ contains $W'$ when $j<(r+1)/2$, and $W'=W$ when $j=(r+1)/2$, it
remains only to determine $W$.  We know $\C{V(W)}$, since we know the rest of
the tree.  If $\C{V(W)}=\ell+1$, then $W$ is just a path and we know it, so we
may assume $\C{V(W)}>\ell+1$.

By Lemma~\ref{kvine}, every $k$-vine in $T$ has fewer than $n-\ell$ vertices
and hence appears as a csc.  Since $\ell+1\le k$, this also holds for every
$(\ell+1)$-vine.  Since $T$ has only one spi-center, there is a unique largest
$(\ell+1)$-vine containing $\Sp{\ell+1}$; it has center $z$.  In this
$(\ell+1)$-vine we see all vertices at distance $\ell+1$ from $z$ in $T$; let
$n'$ be the number of them.

Let $\cC$ be the family of cscs consisting of a path with $2\ell+3$ vertices
and offshoots from the center of the path, at least one of which has length at
least $\ell+1$.  Every such csc contains $\Sp{\ell+1}$, and the branch vertex
of the copy of $\Sp{\ell+1}$ must always be $z$.

Since $r\ge3\ell+3$, any csc in $\cC$ in which two path legs of length
$\ell+1$ come from $T_1$ and $T_2$ omits at least $\ell$ vertices of $P'$
and hence is contained in a card.  Some of these consist of $W$ plus
$2\ell+3$ vertices from $P'$; we will find such a csc.

Let $n_0$, $n_1$, and $n_2$ be the numbers of vertices at distance $\ell+1$
from $z$ in $W$, $T_1$, and $T_2$, respectively.  We know $n_1$ and $n_2$ from
$T_1$ and $T_2$, and we compute $n_0=n'-n_1-n_2$.  The number of cscs with
$2\ell+3+\C{V(W)}$ vertices having path legs in $W$ and $T_2$ is $n_0 n_2$
times the number of rooted subtrees of $T_1$ with $\C{V(W)}$ vertices, and we
know those subtrees.  We similarly eliminate the cscs having legs in $W$ and
$T_1$.  There are no members of $\cC$ in which the two path legs of length
$\ell+1$ come from $W$, since $P'$ has at least $n-3\ell$ vertices and this
would give $C_2$ more than $n-\ell$ vertices.  Since $\C{V(W)}>\ell+1$,
we have discarded cscs with $2\ell+3+\C{V(W)}$ vertices only once.  The
remaining $n_1n_2$ members of $\cC$ with $2\ell+3+\C{V(W)}$ vertices show us
$W$.  This completes the proof of the claim.

It therefore suffices to determine $T_1$ and $T_2$.

\smallskip
{\bf Case 1:} {\it $j<(r+1)/2$ and $C_2$ is asymmetric.}
By asymmetric, we mean that no automorphism of $C_2$ reverses the indexing of
$P'$.  That is, in $C_2$ the vertices $v_j$ and $v_{r+1-j}$ are distinguishable,
although we do not yet know which is $z$.

Recall that $C_2$ has fewer than $n-\ell$ vertices.
Consider augmentations of $C_2$ obtained by adding a path leaving one of
$\{v_j,v_{r+1-j}\}$.  Let $m$ be the maximum $i$ such that we can obtain cscs
by adding to $C_2$ a path of $i$ vertices from either $v_j$ or $v_{r+1-j}$.
Since $j-1\ge\ell+1$ and $z$ is the only spi-center in $T$, no offshoot from
$z'$ has length at least $\ell+1$, so $m\le \ell$.  Since $T$ has a path of
$\ell+1$ vertices grown from $P'$ at $z$, we can tell which of
$\{v_j,v_{r+1-j}\}$ is $z$ by considering trees obtained from $C_2$ by growing
a path of $m+1$ vertices from $v_j$ or $v_{r+1-j}$.  The one that exists as a
csc fixes $z$, since $C_2$ is asymmetric.  By reversing the indexing of $C_2$
along $P'$ if needed, we may now assume $z=v_j$.

Let $C_3$ be a largest subtree of $T$ containing $C_2$ in which $v_j$ has
degree $2$.  Since $C_3$ omits at least $\ell+1$ vertices in an offshoot at
$z$, we see all of $C_3$ in a card, and $C_3$ contains all offshoots from
$P'$ at vertices other than $z$ in full.  This determines $T_1$ and $T_2$, and
the Claim applies to complete the reconstruction of $T$.

\smallskip
{\bf Case 2:} {\it $j<(r+1)/2$ and $C_2$ is symmetric.}
By the symmetry of $C_2$ we may assume $z=v_j$, but we do not yet know the
offshoots at $z$ (that is, $W$) or at $z'$ (that is, $W'$).

\smallskip
{\bf Subcase 2a:} {\it $r+1-2j\le \ell+1$.}
Let $t=\max\{j,r-2j+1+\ell+1\}$.  Let $C_4$ be a largest subtree containing
$S_{\ell+1,\ell+1,t}$ plus offshoots from the spider at the vertex $w$ in the
long leg having distance $r+1-2j$ from the branch vertex.  Since
$\ell+1\le j-1$, we have $j+t\le r$, and hence $T$ contains
$S_{\ell+1,\ell+1,t}$ with the long leg of the spider along $P'$.
Thus $C_4$ exists.

Since $T$ has only one spi-center, the branch vertex of $S_{\ell+1,\ell+1,t}$
in $C_4$ is $z$.  Since $t>j-1$ and $P'$ is a longest path, the long leg of the
spider must extend along $P'$ from $z$.  Again since there is only one
spi-center, the leg cannot depart $P'$ between $z$ and $z'$, so $w=z'$.  Since
the leg extends at least $\ell+1$ vertices past $z'$ and there are at most
$\ell$ vertices in offshoots from $P'$ not at $z$, the leg must continue along
$P'$ past $z'$, and the offshoots from the spider at $w$ must lie in $W'$.

Since $W'$ has at most $\ell$ vertices, and $r+1-2j\le\ell+1$, we have
$t\le \max\{j,2\ell+2\}$.  Thus $C_4$ has at most $\max\{j+3\ell+3,5\ell+5$
vertices.  If $j\le 2\ell$, then since $n\ge 6\ell+5$, the tree $C_4$ fits in a
card, and we see it in the deck.  If $j>2\ell$, then with $T_1$ containing
from $C_4$ at most a short leg of the spider, $C_4$ omits at least $\ell$
vertices from $T_1$ and again fits in a card.  By its maximality, $C_4$ shows
us all of $W'$.  Now we know $T_1$ and $T_2$, and the Claim applies to complete
the reconstruction of $T$.

\smallskip
{\bf Subcase 2b:} {\it $r+1-2j\ge \ell+2$.}
Let $C_5$ be a largest subtree containing $S_{\ell+1,j-1,j}$ such that the end
of the leg of length $j$ in the spider is a leaf of $C_5$.  Since $\ell+1\le
j-1$ and $T$ has only one spi-center, the branch vertex of the required spider
in $C_5$ must be $z$.  All paths of length at least $j$ from $z$ lie in $T_2$.
Since $r-j\ge \ell+j+1$, the subtree $C_5$ omits more than $\ell$ vertices from
the end of $P'$ and hence appears in a card.  Whether the long leg of the spider
in $C_5$ extends into $W'$ or not, by its maximality $C_5$ shows us $T_1$ and
all offshoots in $W$ (it is possible that $W$ has length $j-1$, making $T_1$
and $W$ confusable, but we know them both).

From $C_2$, we know all of $T_2$ except $W'$.  Knowing $W$, we know $\C{V(W)}$.
Since we also know $C_2$, we know how many vertices remain for $W'$; let $m$
denote this value.  Now consider subtrees of $T$ containing an $r$-path
$\VEC u1r$ such that $u_j$ has degree $2$, the offshoots from $u_{r+1-j}$ total
$m$ vertices, and the offshoots from other vertices along the path total
$n-m-\C{V(W)}$ vertices.  Since such subtrees with $u_j=v_j$ omit at least a
path of $\ell+1$ vertices from $W$, they fit in cards, so we see them.  Since
we know $W$, we know all such subtrees in which $u_{r+1-j}=z$.  Discarding them
leaves a csc that shows us $W'$, completing the reconstruction.

{\bf Case 3:} {\it $j=(r+1)/2$.}
Here the primary vertex $z$ of $C_1$ is the central vertex of $P'$.  The csc
$C_2$ shows us all of $T$ except the offshoots at $z$; in particular, we know
$T_1$ and $T_2$.  The Claim applies to complete the reconstruction of $T$.
\end{proof}

\begin{lemma}\label{twolctr}
If $n\ge6\ell+6$ and $r\ge n-3\ell$, and $T$ contains both $\Sp{\ell+1}$ and a
sparse card, then $T$ is $\ell$-reconstructible.
\end{lemma}
\begin{proof}
The hypotheses of this lemma are the same as in Lemma~\ref{onelctr}, except
for dropping the restriction that $T$ has only one spi-center.  Hence all
steps of Lemma~\ref{onelctr} until we first use the restriction on spi-centers
remain valid.  In particular, we have the sparse card $C_1$ containing
$r$-vertex path $P$ with primary branch vertex $z$, recognize $k\ge\ell+1$, and
see $\Sp{\ell+1}$ in $T$.  Again, $z$ is a spi-center, and any other spi-center
is in $V(C_1)-V(P)$.  Again, we define $j$ with $j\le(r+1)/2$ by $z=v_j$ with
$P$ indexed as $\VEC v1r$.  By Lemma~\ref{central}, all $r$-vertex paths have
the same central vertices $\VEC vj{r+1-j}$.

Lemma~\ref{onelctr} handles the case when $T$ has only one spi-center.  Since
by Corollary~\ref{spi} we know the number of spi-centers, we recognize that we
are not in that case.  Hence we may assume that $C_1$ has at least one vertex
in an offshoot from $P$ at $z$ that is a spi-center.  Let $x$ be one such
vertex.  Because $P$ is a longest path and some path from $z$ extends along a
leg of $\Sp{\ell+1}$ beyond $x$, we have $j-1\ge \ell+2$.

\smallskip
{\bf Step 1:} {\it $z$ and $x$ are the only spi-centers, and we know the
distance between them.}
Consider the process where we begin with $P$ and iteratively complete a copy of
$\Sp{\ell+1}$ at a new spi-center of $T$.  Each time we do this, we raise the
degree of the new
spi-center $v$ from $2$ to $3$ and hence must add $\ell+1$ vertices in a path
emanating from $v$.  Hence if $T$ has $s$ spi-centers, then $T$ has at least
$r+s(\ell+1)$ vertices.  With $r\ge n-3\ell$, we thus have
$n\ge n-(3-s)\ell+s$, so $s\le 2$.

Now consider this computation more closely.  Let $p$ be the length of the path
from $z$ to $x$.  When adding $\ell+1$ vertices off $P$ for the copy of
$\Sp{\ell+1}$ centered at $z$, we may treat them as beyond $x$, after the $p$
vertices in $T$ on the path from $z$ to $x$.  Thus the union of $P$ and two
copies of $\Sp{\ell+1}$ centered at $z$ and $x$ has $n-\ell+2+p$ vertices.  We
conclude $p\le \ell-2$.

The union $X$ of copies of $\Sp{\ell+1}$ at $z$ and $x$ consists of four legs
of length $\ell+1$ plus $p+1$ vertices in the path joining $z$ and $x$.
Since $p\le \ell-2$, we have $\C{V(X)}\le 5\ell+3$.  Since $n\ge 6\ell+3$,
we see $X$ in a card, and thus we know $p$.

\smallskip
{\bf Step 2:}
{\it Reconstructing $T_2$ and the component of $T-z$ containing $x$, which we
call $Q$.}
Let $C'$ be a largest subtree of $T$ containing $S_{\ell+1,\ell+1,p+r-j}$
with branch vertex $x'$ such that the vertex $z'$ at distance $p$ from $x'$
along the long branch of the spider has degree $2$ in $C'$.  There is such a
subtree with $x'=x$ and $z'=z$.  Indeed, since $x'$ is a spi-center, we must
have $x'\in\{x,z\}$; since no path of length more than $r-j$ starts from $z$,
we have $x'=x$.  Furthermore, a path of length $p+r-j$ from $x$ must pass 
through $z$ to avoid having a path longer than $P$, so $z'=z$.

Since $z$ has degree $2$ in $C'$, at least $j-1$ vertices of $P$ are missing
from $C'$.  Since $j-1\ge\ell+2$, we know that $C'$ fits in a card.  The
components of $C'-z$ are $T_2$ and the offshoot from $z$ that contains $x$,
which we call $Q$.  We have not yet determined other offshoots from $z$,
such as those contributing to the sparse card $C_1$ if $z$ has degree more
than $3$ in $C_1$, but we now know $T_2$ and $Q$.

Note that if $j=(r+1)/2$, then in $C'$ we see the larger of $T_1$ and $T_2$,
which we call $T_2$, but if they have the same size then there are two choices
for $C'$ and we obtain both.  There cannot be another offshoot of length
$(r-1)/2$ from $z$ besides $T_1$, $T_2$, and $Q$, because $r\ge n-3\ell$
implies $3(r-1)/2+2\ell+4>n$.

\smallskip
{\bf Step 3:}
{\it Reconstructing the components of $T-z$ other than $T_2$ and $Q$.} 
Note that $Q$ has at least $2\ell+3$ vertices.  Since $r\ge n-3\ell$, any
offshoot from $P$ other than $Q$ has fewer than $\ell$ vertices and hence also
length less than $j-1$.  Let $C''$ be a largest long subtree containing
$S_{\ell+1,j-1,r-j}$ that has no offshoots from the legs of leg $\ell+1$
and $j-1$.  In such a subtree, the branch vertex of the given spider must
be $z$, and more than $\ell$ vertices are missing from $Q\cup T_1$.
Hence $C''$ shows us all of $T$ outside $Q\cup T_1$.

Since we also know $Q$, it remains only to determine $T_1$.  If $Q$ has length
less than $j-1$, then every $r$-path lies in $T_1\cup T_2$.  Enlarge $C''$
to a largest subtree containing $S_{\ell+1,j-1,r-j}$ that has no offshoot
from the leg of length $\ell+1$.  The leg of length $j-1$ must lie in $T_1$.
Again more than $\ell$ vertices are missing from $Q$ and the subtree fits
in a card, showing us $T_1$ (if $j-1=r-j$, then we already knew $T_1$).

The remaining case is that $Q$ has length $j-1$, making it unclear whether
the leg of length $j-1$ in the spider in $C''$ lies in $Q$ or in $T_1$.
Nevertheless, $C''$ shows us all components of $T-z$ other than $Q$ and $T_1$,
and we already know $Q$ from $C'$.  Hence we know $\C{V(T_1)}$.  We also know
the number of $r$-paths, so knowing $Q$ and $T_2$ we can compute the number
of peripheral vertices in $T_1$.

Now let $\cC$ be the family of subtrees containing a vertex $z'$ of degree $2$,
such that emerging from $z'$ are one path of length $r-j$ and one offshoot $Y$
of length $j-1$ that need not be a path.  Note that $z'$ may be $z$ or
$v_{r+1-j}$.  Since we know $Q$, $T_2$, and the numbers of peripheral vertices
in $Q$, $T_2$, and $T_1$, we know all the members of $\cC$ in which $Y$ comes
from $Q$ or $T_2$.  After excluding all of these members of $\cC$, a largest
remaining member of $\cC$ shows us $T_1$.
\end{proof}

\begin{lemma}\label{spidone}
If $n>6\ell+6$ and $r\ge n-3\ell$, and $T$ contains $\Sp{\ell+1}$
but no sparse card, then $T$ is $\ell$-reconstructible.
\end{lemma}
\begin{proof}
By Lemma~\ref{longp}, we know $r$ from the deck.  As noted in
Lemma~\ref{largek}, $r\ge3\ell+6$ yields $k\ge\ell+1$.  By Corollary~\ref{spi},
we then know the number of spi-centers in $T$, so we know $\Sp{\ell+1}\esub T$
(and hence $r<n-\ell$) and recognize from the deck that $T$ has no sparse card.

If some spi-center is outside an $r$-path, then it forces at least $2\ell+3$
vertices in one offshoot from the $r$-path, yielding a sparse card when
$r\ge n-3\ell$.  With sparse cards forbidden and $\Sp{\ell+1}$ present, we
have a spi-center on an $r$-path.  Among subtrees having an $r$-path $P$ with a
spi-center closest to the center and no other branch vertex on $P$, choose
$C_1$ to be one having a longest offshoot from $P$ and, subject to this, having
the most vertices.  Let $z$ be the spi-center in $C_1$.  There may be more than
one largest such card; fix one such as $C_1$.

With $V(P)$ indexed as $\VEC v1r$, define $j$ with $j\le(r+1)/2$ by
$z\in\{v_j,v_{r+1-j}\}$.  We have chosen $C_1$ so that $\{\VEC v{j+1}{r-j}\}$
contains no spi-center, so all $r$-vertex paths contain $\VEC vj{r+1-j}$, by
Lemma~\ref{central}.  Since $T$ has no sparse card, $C_1$ has fewer than
$n-\ell$ vertices, and in $C_1$ we see all offshoots from $P$ at $z$ in $T$.
Since $C_1$ has fewer than $n-\ell$ vertices and $r\ge n-3\ell$, there is
exactly one offshoot $Q$ from $P$ at $z$ that has length at least $\ell+1$.
Also let $T_1$ and $T_2$ be the components of $T-z$ containing $v_{j-1}$
and $v_{j+1}$, respectively.  In $C_1$ we see in full all offshoots from $z$ in
$T$ other than $T_1$ and $T_2$; from $T_1$ and $T_2$ we see only what is in
the $r$-path $P$.

Let $C_2$ be a largest subtree containing an $r$-path $P'$ with vertices
$\VEC u1r$ on which the vertices $u_j$ and $u_{r+1-j}$ have degree $2$.  We
have noted that all $r$-paths have the same $r-2j+1$ central vertices.
Thus $u_j\in\{v_j,v_{r+1-j}\}$.  Since $z$ is a spi-center and has degree $2$
in $C_2$, the omitted offshoots at $z$ contain at least $\ell+1$ vertices.
Thus $C_2$ fits in a card, and we see $C_2$ as a largest such csc.  Since
we see also the offshoots from $P'$, in $C_2$ we may assume $P'=P$, and in $C_2$
we see all offshoots from $P$ except at $v_j$ and $v_{r+1-j}$.

\medskip
{\bf Case 1:} {\it The length of $Q$ is less than $j-1$.}
If $j=(r+1)/2$, then $C_1$ tells us the offshoots from $C_2$ at $v_j$, and we
have reconstructed $T$.  Hence we may assume $j<(r+1)/2$.  We know the
offshoots from $P$ at $z$, but we do not yet know whether $z$ is $v_j$ or
$v_{r+1-j}$.  Let $y$ be the vertex of $\{v_j,v_{r+1-j}\}$ other than $z$.

Let $W_1$ be the union of the offshoots from $P$ at $z$, let $W_2$ be the union
of the offshoots at $y$, and let $m_i=\C{V(W_i)}$.  We already know $W_1$ and
$m_1$.  We also know $m_2$, since we know $r$ and all offshoots from $P$
outside $W_2$.  We want to find $W_2$ and decide which of $\{W_1,W_2\}$ is
attached to $v_j$. 

Suppose first that $m_2\ge m_1$.  Note that $y$ may or may not be a spi-center.
Let $C_3$ be a largest subtree containing an $r$-vertex path
on which one of the vertices in positions $j$ and $r+1-j$ has degree $2$.
Since $C_3$ omits $m_1$ vertices, $C_3$ has fewer than $n-\ell$ vertices,
and we see $C_3$ in a card as a largest such csc.  If $m_2>m_1$, then we
see $W_2$ in position along $P$, together with all offshoots from $P$ except
those in $W_1$.  We attach $W_1$ to the vertex in $\{v_j,v_{r+1-j}\}$ having
degree $2$ in $C_3$ to complete the reconstruction of $T$.

If $m_2=m_1$, then there are two choices for $C_3$, one showing $W_2$ and
the other showing $W_1$, each occurring in the right position among the
offshoots from $P$.  We may have $W_1=W_2$, and even the two choices
for $C_3$ may be the same, but in all cases we complete the reconstruction.

Now suppose $m_2<m_1$.  In this case let $C_3$ be a largest subtree containing
an $r$-vertex path $\VEC u1r$ such that $u_j$ or $u_{r+1-j}$ has degree $2$,
and such that offshoots with a total of $m_2+1$ vertices are grown from the
other of $\{u_j,u_{r+1-j}\}$.  Since $W_1\cup W_2$ has $m_1+m_2$ vertices,
$C_3$ omits $m_1-1$ vertices from $W_1\cup W_2$.  Since $m_1\ge\ell+1$, $C_3$
fits in a card, and we see it as a largest such csc.  Since $W_2$ has only
$m_2$ vertices, $C_3$ determines which of $\{v_j,v_{r+1-j}\}$ is $z$.

We now know all of $T$ except the offshoots from $P$ at $y$.  Let $\cC_4$ be
the family of largest subtrees containing an $r$-vertex path $\la \VEC u1r\ra$
such that one of $u_j$ and $u_{r+1-j}$ has degree $2$ and the other has
offshoots from the path with a total of $m_2$ vertices.  Such subtrees omit
$m_1$ vertices and hence are visible as cscs.  They may arise from $T$ by
deleting $m_1-m_2$ vertices from $W_1$ and all of $W_2$, or by deleting
all of $W_1$.  Since we know $W_1$ and where in $C_2$ it is attached, we know
all the members of $\cC_4$ that arise in the first way.  The remaining member
of $\cC_4$ shows us $W_2$, completing the reconstruction of $T$.

\medskip
{\bf Case 2:} {\it $Q$ has length $j-1$.}
In $C_1$ we see $Q$.  Since $Q$ has length $j-1$ and $C_1$ was chosen to make
$Q$ largest among such offshoots, there is a choice for $C_2$ such that $Q$ is
a component of $C_2-\{v_j,v_{r+1-j}$ containing an endpoint of $P$.  Possibly
we can choose $C_2$ so that both such components are $Q$.  We choose $C_2$
to maximize the number of such occurrences of $Q$.

Let $W_1$ be the union of the offshoots in $T$ from $C_2$ at $z$; we do not yet
know $W_1$.  Since $T$ has no sparse card, $W_1$ contains exactly one offshoot
from $C_2$ with length at least $\ell+1$; call it $Q'$.  Since $C_1$ had only
$Q$, $T_1$, and $T_2$ as offshoots from $z$ with length at least $\ell+1$, the
choice of $C_2$ in fact implies $Q'=T_1$.  Let $W_2$ be the union of the
offshoots from $C_2$ at the vertex in $\{v_j,v_{r+1-j}\}$ other than $z$, and
let $m_i=\C{V(W_i)}$.

Since $z$ is a spi-center, $m_1\ge m'\ge\ell+1$.  We may assume that the path
is indexed so that $z=v_j$ if $C_2$ is symmetric or if $C_2$ is not symmetric
and shows only one copy of $Q$ giving a candidate for $z$.

{\bf Subcase 2a:}
{\it $C_2$ is not symmetric (under reversal of $\VEC v1r$).}
In $C_2$ we see at least one copy of $Q$, rooted at $v_j$ or $v_{r+1-j}$.
Let $C_3$ be a largest csc containing $C_2$ in which a root of $Q$ has degree
$2$.  Since $C_2$ is not symmetric, the vertices $v_j$ and $v_{r+1-j}$ are
distinguished by what we see along $C_2$, whether or not they are both roots of
copies of $Q$.  Hence in addition to $C_3$ we can obtain a largest csc $C'_3$
containing $C_2$ where the added vertices are in offshoots grown from the other
of these two vertices.  Since $m_1\ge\ell+1$, there are fewer than
$n-\ell$ vertices in $C_3'$, so it gives us $W_2$ in full.  If also $C_3$ has
fewer than $n-\ell$ vertices, then we obtain both $W_1$ and $W_2$ with their
roots along $C_2$.  This completes the reconstruction, since $C_2$ shows us
the rest of $T_2$ by showing the offshoots at the vertices other than
$v_{r+1-j}$.

Hence we may assume that $C_3$ is a card, which requires $m_2\le\ell$,
and we know $W_2$.  We may assume the indexing so that $z=v_j$.  From what we
see of $T_2$ in $C$, plus $W_2$, we know $T_2$.  Subtracting $\C{V(T_2)}$ and
$\C{V(Q)}$ from $n-1$ tells us $m_1$.  Subtracting the sizes of the small
offshoots at $z$ (in $C_1$) from $m_1$ now gives us $m'$, which is $\C{V(T_1)}$.

To find $T_1$, let $\cC_4$ be the family of subtrees consisting of an $r$-vertex
path plus one offshoot of length $j-1$ from the $j$th vertex, such that the
offshoot has $m'$ vertices.  Since $m'\le \C{V(Q)}$, some of these arise by
deleting vertices from $Q$; others show $Q'$.  Since we know $Q$, we know the
number of ways to find in $Q$ an offshoot of length $j-1$ with $m'$ vertices.
If we can determine the number of vertices of $Q'$ at distance $j-1$ from $z$
and the number of vertices at distance $r-j$ from $z$ (we see the latter in
$C_3$), then we know all the members of $\cC_4$ to delete.  In the remaining
members we see $T_1$, completing the reconstruction of $T$.

Consider the subtree $C_5$ consisting of an $r$-vertex path $\VEC u1r$ plus one
offshoot isomorphic to $Q$ at $u_j$.  Since $T$ has no sparse card, $C_5$ has
fewer than $n-\ell$ vertices, and we can count its appearances in $T$.  By the
asymmetry of $C_2$, we have $u_j=z$.  Since we know $W_2$, we know how many
vertices can serve as $u_r$.  If $Q'\ne Q$, then the number of vertices that
can serve as $u_1$ is the number of vertices of $Q'$ at distance $j-1$ from $z$.
We can thus determine this quantity from the number of copies of $C_5$.

If $Q'=Q$, then we had more than one choice for our original subtree $C_1$,
and they were identical, so we knew $Q'$ all along.

{\bf Subcase 2b:} {\it $C_2$ is symmetric (under reversal of $\VEC v1r$).}
In this case, we cannot distinguish $v_j$ and $v_{r+1-j}$ in $C_2$.  We still
consider the family $\cC_3$ of subtrees containing $C_2$ in which $v_j$ or
$v_{r+1-j}$ has degree $2$.  If no csc in $\cC_3$ is a card, then let
$C_3$ be a largest such csc among those having an offshoot of length $j-1$ from
$C_2$.  We may assume that the offshoot is at $v_j$, showing us $W_1$.  Since
$W_1$ tells us $m_1$ and we know $C_2$, we now also know $m_2$.  If $m_2>m_1$,
then a largest member of $\cC_3$ shows us $W_2$.  If $m_2\le m_1$, then we know
all the members of $\cC_3$ having $m_2$ vertices outside $C_2$ that arise by
deleting $m_1-m_2$ vertices from $W_1$ in $C_3$.  The remaining member of
$\cC_3$ with $m_2$ vertices outside $C_2$ shows us $W_2$.

Hence we may assume that $\cC_3$ contains a card, so $m_2\le\ell$.  Since $W_1$
has at least $\ell+1$ vertices, the offshoots from $C_2$ in cscs that are cards
in $\cC_3$ are attached at $v_j$.

We first consider the case $j<(r+1)/2$, postponing $j=(r+1)/2$.  Let the
components of $T-z$ containing $v_{j-1}$ and $v_{j+1}$ be $T_1$ and $T_2$,
respectively.  Since $z$ lies in all longest paths as the $j$th vertex from one
end, and $m_2\le\ell$, each copy of $S_{j-1,\ell+1,r-j}$ in $T$ has branch
vertex $z$ and one leaf in each of $Q$, $T_1$, and $T_2$.  Since $z$ is a
spi-center, $\ell+1\le j-1$, so the same conclusion holds for $S_{j-1,j-1,r-j}$.

Thus the number of copies of $S_{j-1,j-1,r-j}$ in $T$ is $t_Qt_1t_2$, where
these factors are the numbers of vertices at distance $j-1$ from $z$ in $Q$, at
distance $j-1$ from $z$ in $T_1$, and at distance $r-j$ from $z$ in $T_2$,
respectively.  Since $T$ has no sparse cards, these spiders fit into cards, so
we know the number of them.  Knowing $Q$ from $C_1$, we know $t_Q$.  Since
$C_2$ is symmetric and $m_2\le \ell$, we have $t_2=t_Q$.  Hence we can compute
$t_1$.

As noted, the same properties hold for the smaller spider $S_{j-1,\ell+1,r-j}$.
Again we can count them in the deck, and the number of them is computed as
$t_Qs_1t_2+s_Qt_1t_2$, where $s_Q$ and $s_1$ are the numbers of vertices
at distance $\ell+1$ from $z$ in $Q$ and in $T_1$, respectively.  Since we 
know $Q$, we know all these numbers except $s_1$, so now we also know $s_1$.

Now let $\cC$ be the family of cscs containing $S_{j-1,\ell+1,r-j}$ such that
the leg of length $r-j$ and at least one leg of length $\ell+1$ have no
offshoots ($j-1=\ell+1$ is possible).  Again, the branch vertex of the spider
must be $z$.  Largest cscs in $\cC$ may or may not be cards, but in either case
we know all members of $\cC$ that arise by choosing the leg of length $\ell+1$
without offshoots from $T_1$ and the leg of length $j-1$ with offshoots from
$Q$, since we know $s_1$ and $Q$.  A largest csc among the remaining members
of $\cC$ contains a leg of length $\ell+1$ from $Q$ and shows all of $T_1$.

Now we know all of $T$ except $W_2$, so we know $m_2$.  We return to the family
$\cC_3$.  Since we know $T_1$, we know which members of $\cC_3$ among those
having $m_2$ vertices in addition to $C_2$ are obtained from $C_2$ by adding
$m_2$ vertices from $T_1$.  After eliminating them, a remaining member of
$\cC_3$ shows us $W_2$.

Finally, we have the case $j=(r+1)/2$ with $\cC_3$ containing a card.
There are three offshoots from $z$ with length $j-1$: two copies of $Q$ and
one of $T_1$, and we know $Q$ and the shorter offshoots.  The copies
of $\Sp{j-1}$ have branch vertex at $z$ and leaves in $T_1$ and the two copies
of $Q$.  We count them and know $t_Q$, so we obtain $t_1$ (as defined above).
We can make use of copies of $S_{j-1,\ell+1,r-j}$ and cscs containing them
as above to obtain $T_1$.
\end{proof}

Lemmas~\ref{onelctr}--\ref{spidone} cover all cases with
$n\ge6\ell+6$ and $r\ge n-3\ell$ in which $\Sp{\ell+1}\esub T$,
so we have proved that such trees are $\ell$-reconstructible

\section{The Vines at the Ends of $P$}

In the previous section we proved the $\ell$-reconstructibility of $n$-vertex
trees $T$ with $n\ge6\ell+6$ under the conditions $r\ge n-3\ell$ and
$\Sp{\ell+1}\esub T$.  We next consider trees not containing $\Sp{\ell+1}$.
Thus forbidding $\Sp{\ell+1}$ will also forbid $\Sp{k-1}$ when $k\ge\ell+2$.
We obtain $k\ge\ell+2$ using the computation in Lemma~\ref{largek} and the
fact that $k$ is an integer as soon as $r\ge 3\ell+7$, so here we raise the
threshold for $n$ to $n\ge6\ell+7$.

\begin{lemma}\label{t1t2}
When $\Sp{k-1}\nosub T$, in every reconstruction $T$ all longest paths have the
same $r-2k+2$ central vertices, which are all the $(k-1)$-centers in $T$.  The
central $r-2k$ vertices are all the $k$-centers in $T$.  Thus the two vertices
at distance $k-1$ from the ends of every longest path are $(k-1)$-centers but
not $k$-centers.  From the deck, we can determine the two maximal $(k-1)$-vines
$U_1$ and $U_2$ centered at these two vertices.
\end{lemma}
\begin{proof}
A $(j-1)$-center $v$ outside the central $r-2j+2$ vertices of a longest path $P$
yields either a longer path than $P$ (if $v\in V(P)$) or a copy of $\Sp{j}$
(if $v\notin V(P)$).  Hence $\Sp{j}\nosub T$ implies that $T$ has no
$(j-1)$-center outside the central $r-2j+2$ vertices of any longest path.
Those central vertices are indeed $(j-1)$-centers, so all longest paths have
the same central $r-2j+2$ vertices.

Considering $j=k$ and $j=k-1$, we find that the $k$-centers are precisely the
vertices of $P$ except for the last $k$ vertices on each end, and the
$(k-1)$-centers are the vertices of $P$ except for the last $k-1$ vertices on
each end.  In particular, the vertices at distance $k-1$ from the ends of $P$
are $(k-1)$-centers but not $k$-centers.

By Corollary~\ref{countk}, we know all the maximal $k$-vines and all the
maximal $(k-1)$-vines.  Each maximal $k$-vine has a unique center and contains
exactly one maximal $(k-1)$-vine having the same center.  Hence we can
eliminate the maximal $(k-1)$-vines contained in $k$-vines from the list
of all maximal $(k-1)$-vines to leave only the two maximal $(k-1)$-vines whose
centers have distance $k-1$ from the ends of $P$.  These two $(k-1)$-vines
$U_1$ and $U_2$ contain the opposite ends of $P$.
\end{proof}

\vspace{-1pc}

Reconstructing $T$ requires assembling $U_1$ and $U_2$ and any part of $T$
omitted by them.  Let $x_i$ be the center of $U_i$.  Since $\Sp{k-1}\nosub T$,
exactly two edges incident to $x_i$ in $U_i$ start paths of length $k-1$ in
$U_i$.  Our next task is to determine which of these two edges incident to
$x_i$ in $U_i$ starts the path to $x_{3-i}$.  By ``orienting'' $U_i$, we mean
determining which of these two edges in $U_i$ lies along the $x_1,x_2$-path in
$P$ (here $P$ is any longest path).  To facilitate this task we
introduce definitions and notation for various objects in the tree.

\begin{definition}\label{trunketc}
{\it Structure of $U_i$.}  See Figure~\ref{tifig}.
The edge of $U_i$ incident to $x_i$ that belongs to the $x_1,x_2$-path in
$P$ is the {\it trunk edge} of $U_i$.  The trunk edge of $U_i$ is the central
edge in a maximal $(k-1)$-evine.  The other edge of $P$ incident to $x_i$
in $U_i$ is not the central edge of a $(k-1)$-evine in $T$.

As noted above, $\Sp{k-1}\nosub T$ implies that exactly two components of
$U_i-x_i$ contain $(k-1)$-vertex paths beginning at the neighbor of $x_i$.
Let $A_i$ and $B_i$ respectively be the subtrees of $U_i$ rooted at $x_i$ that
are obtained by deleting the vertex sets of those two components.  Both $A_i$
and $B_i$ contain all components of $U_i-x_i$ having no $(k-1)$-vertex path
beginning at the neighbor of $x_i$.  Together with $x_i$, these comprise
$A_i\cap B_i$, and we call this rooted subtree $W_i$.
%
\end{definition}

\begin{figure}[h]
\begin{center}
\includegraphics[scale=0.5]{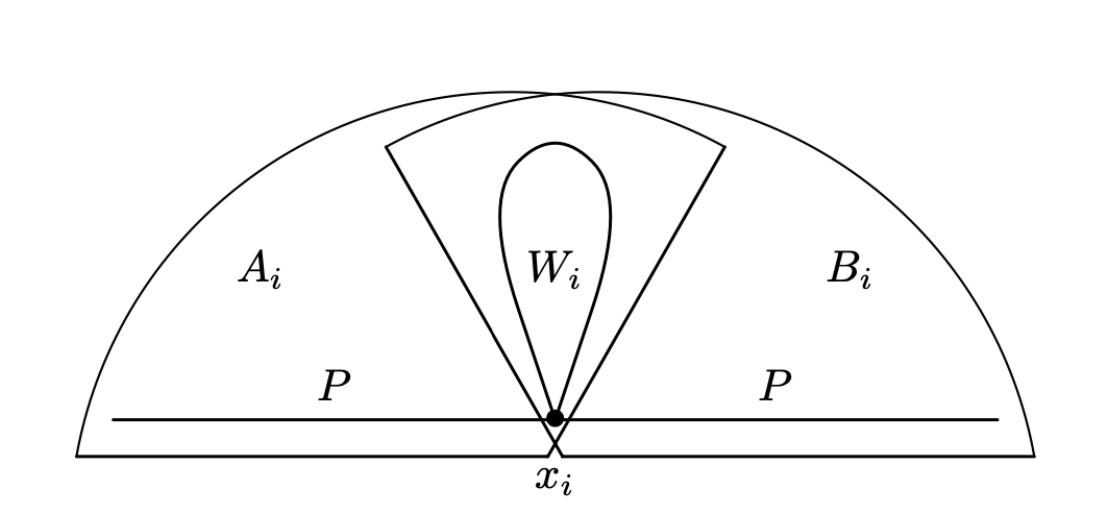}
\caption{$U_i$}\label{tifig}
\end{center}
\end{figure}

\vspace{-1pc}

\begin{definition}\label{Sj}
{\it $(k-1)$-evines containing $U_1$ and $U_2$.}
By Corollary~\ref{countk}, we can determine the number of maximal
$(k-1)$-evines and maximal $k$-evines in $T$ with each isomorphism type, and
each maximal $k$-evine contains exactly one maximal $(k-1)$-evine with the same
central edge.  Therefore, just as we determined $U_1$ and $U_2$ in
Lemma~\ref{t1t2}, we can also determine the maximal $(k-1)$-evines $S_1$ and
$S_2$ whose central edges are the trunk edges in $U_1$ and $U_2$.
From the list of maximal $k$-evines, we obtain a list of $(k-1)$-evines by
deleting the vertices at distance $k$ from the central edge.  The members of
the list of maximal $(k-1)$-evines that are not generated in this process
(paying attention to multiplicity) are $S_1$ and $S_2$.  However, we do not
yet know which of $\{S_1,S_2\}$ contains which of $\{U_1,U_2\}$.

For $j\in\{1,2\}$, let $y_jz_j$ be the central edge of $S_j$.  Let $C_j$ and
$D_j$ be the components of $S_j-y_jz_j$, rooted at $y_j$ and $z_j$,
respectively.  Since $\Sp{k-1}\nosub T$, exactly two
components of $S_j-\{y_j,z_j\}$ contain paths with $k-1$ vertices starting from
the neighbor of $\{y_j,z_j\}$.  One is in $C_j$ and one in $D_j$; these are
the {\it major pieces} of $C_j$ and $D_j$; see Figure~\ref{sjfig}.

\begin{figure}[h]
\begin{center}
\includegraphics[scale=0.45]{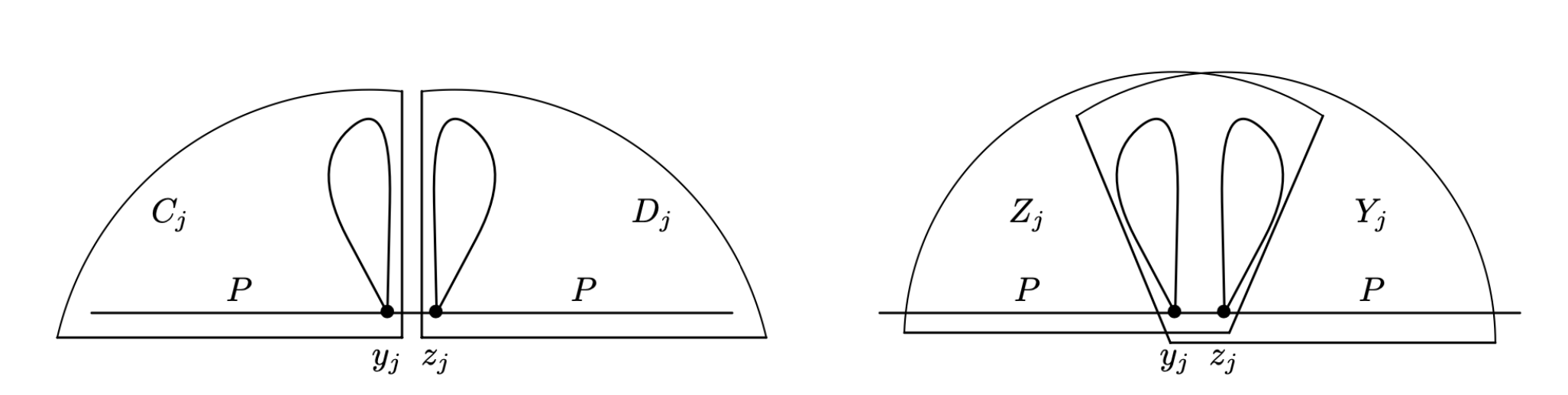}
\caption{Rooted trees in $S_j$}\label{sjfig}
\end{center}
\vspace{-1pc}
\end{figure}

The $(k-1)$-evine $S_j$ contains two maximal $(k-1)$-vines, centered at
$y_j$ and $z_j$.  Let $Y_j$ be the tree rooted at $y_j$ obtained by deleting
the major piece of $C_j$ from the maximal $(k-1)$-vine centered at $y_j$.
Let $Z_j$ be the tree rooted at $z_j$ obtained by deleting the major piece of
$D_j$ from the maximal $(k-1)$-vine centered at $z_j$.  Thus the two maximal
$(k-1)$-vines contained in $S_j$ are $C_j\cup Y_j$ (centered at $y_j$) and
$D_j\cup Z_j$ (centered at $z_j$).
Offshoots from $P$ at $y_j$ having length less than $k-1$ lie in both
$C_j$ and $Y_j$, and similarly for $z_j$ and $D_j\cap Z_j$.

The leaves at maximum distance from the root in a rooted subtree $R$ are
the {\it $R$-extremal vertices} or simply the {\it extremal vertices} when
the rooted subtree is clear from context.  We use this language particularly
with these rooted subtrees where the maximum distance from the root is $k-1$.
For example, we can observe that $C_j\cup Y_j$ contains all of $S_j$ except the
extremal vertices of $D_j$, and $D_j\cup Z_j$ contains all of $S_j$ except the
extremal vertices of $C_j$.
\end{definition}

We may speak of two rooted subtrees of $T$ (such as among those we have defined
above) as being ``the same''; this means that those rooted trees
are isomorphic.  In particular, we next develop a tool that obtains some
isomorphisms from others.

In Lemma~\ref{isom}, $S_j$ denotes any $(k-1)$-evine (not necessarily in $T$)
whose branches are $C_j$ and $D_j$ as in Definition~\ref{Sj}, with subtrees
$Y_j$ and $Z_j$ also as in Definition~\ref{Sj}.  In one of the applications of
Lemma~\ref{isom} where we consider the $(k-1)$-evines containing $U_1$ and
$U_2$ in $T$, the subtrees in the lemma will indeed be the subtrees of $T$
having these names in Definition~\ref{Sj}.

\begin{lemma}\label{isom}
Let $S_1$ and $S_2$ be $(k-1)$-evines not containing $\Sp{k-1}$, with rooted
subtrees $(C_1,Z_1,Y_1,D_1)$ and $(C_2,Z_2,Y_2,D_2)$ as in Definition~\ref{Sj}.
If $C_1\cong Y_2$, $Z_1\cong D_2$, $Y_1\cong C_2$, and $D_1\cong Z_2$,
then $C_1\cong Y_1$ and $D_1\cong Z_1$.
\end{lemma}
\begin{proof}
As in Definition~\ref{Sj}, let the roots of $C_1,D_1,C_2,D_2$ be
$y_1,z_1,y_2,z_2$, respectively.  In each of these trees, the extremal
vertices are at distance $k-1$ from the root.  The given pairs are isomorphic
as rooted trees, so the isomorphisms match the roots.

Since $\Sp{k-1}$ does not appear, the common subtree $W_1$ in $C_1$ and $Y_1$
that is rooted at $y_1$ and contains no edge of a fixed longest path in $S_1$
has length less than $k-1$.  Since $C_1\cong Y_2$ and $Y_1\cong C_2$, also the
common subtree of $C_2$ and $Y_2$ rooted at $y_2$ is $W_1$.  The analogous
observation holds for $W_2$ rooted at both $z_1$ and $z_2$.  In
Figure~\ref{s1s2fig}, copies of $W_1$ are drawn rounded, while copies of $W_2$
are drawn as triangles.  We have discussed the copies of $W_1$ and $W_2$ as
offshoots from the central vertices; we will be deriving the copies at the
neighboring vertices.  To facilitate discussion, we write $A_1$ for the
isomorphic subtrees $C_1$ and $Y_2$, $B_1$ for $C_2$ and $Y_1$, $A_2$ for $D_2$
and $Z_1$, and $B_2$ for $D_1$ and $Z_2$, as shown in Figure~\ref{s1s2fig}.

\begin{figure}[h]
\begin{center}
\includegraphics[scale=0.44]{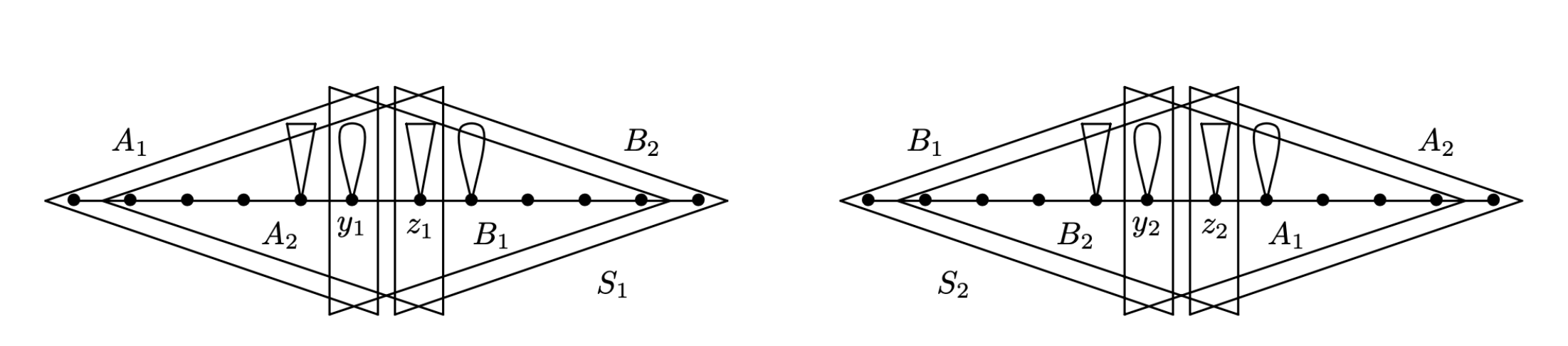}
\caption{$(k-1)$-evines in Lemma~\ref{isom}\label{s1s2fig}}
\end{center}
\end{figure}

Iteratively, moving away from the roots, we will follow paths in the subtrees
$A_1,A_2,B_1,B_2$ so that the offshoots from corresponding vertices in the
paths in $A_1$ and $B_1$ are isomorphic, and similarly for $A_2$ and $B_2$.
Since paths from the roots have finite length, this iteration leads to
$A_1\cong B_1$ and $A_2\cong B_2$, which yields $C_1\cong Y_1$ and
$D_1\cong Z_1$.

At the roots of $A_i$ and $B_i$, we have the common subtree $W_i$ containing
the offshoots of length less than $k-1$.  Consider the subtree $W_2$ rooted at
$z_2$ in $S_2$.  Not only is $z_2$ the root of $Z_2$, which is a copy of $B_2$,
but also $z_2$ is the child outside $W_1$ of the root $y_2$ of $Y_2$, which is
a copy of $A_1$.  Thus the copy of $A_1$ that appears as $C_1$ in $S_1$ must
also have a copy of $W_2$ rooted at the child of $y_1$ not in $W_1$.
Similarly, consider the copy of $W_2$ rooted at $z_1$ in $S_1$, appearing
in $A_2$ and $B_2$.  Since $z_1$ is the child outside $W_1$ of the root
$y_1$ of the copy of $B_1$ in $S_1$, the copy of $B_1$ in $S_2$ also has
$W_2$ rooted at the child of $y_2$ not in $W_1$.  We have thus obtained
isomophism between subtrees at the children of $y_1$ in $A_1$ and $B_1$.

Also, the copy of $A_1$ occurring as $Y_2$ has only one child of $z_2$
outside $W_2$.  Hence for the copy of $A_1$ occurring as $C_1$ in $S_1$,
also there is only one child outside $W_2$ for the vertex whose offshoots we
have found to be $W_2$, thereby leaving a unique path to follow (similarly
for $B_1$.

Analogous arguments starting from the copies of $W_1$ rooted at $y_2$ in $S_2$
and at $y_1$ in $S_1$ give us isomorphism between the subtrees at the child of
the root of $A_2$ in $S_2$ and at the child of the root of $B_2$ in $S_1$.

At the child of the root along the path, we have just found
a copy of $W_2$ in $A_1$ in $S_1$, a copy of $W_1$ in $B_2$ in $S_1$,
a copy of $W_2$ in $B_1$ in $S_2$, and a copy of $W_1$ in $A_2$ in $S_2$,
where these hosts are the ``outer'' subtrees together covering all of $S_1$ and
$S_2$ except the central edges.  Since the roots of these new copies of $W_1$
and $W_2$ are one step from the centers of $S_1$ and $S_2$ in these hosts,
the copies of $W_1$ and $W_2$ may contain vertices that are extremal in these
host subtrees.  These extremal vertices are deleted when we view these
offshoots as subtrees of the ``inner'' subtrees, moving from $A_1$ to $A_2$ in
$S_1$, from $B_2$ to $B_1$ in $S_1$, from $B_1$ to $B_2$ in $S_2$, and from
$A_2$ to $A_1$ in $S_2$.

Since these copies of $W_i$ are isomorphic and at the same distance from the
center, the truncations obtained by deleting vertices that are extremal in the
inner hosts are also isomorphic.  In comparing the inner subtrees after the
truncation ($A_2$ in $S_1$ with $B_2$ in $S_2$, and $B_1$ in $S_1$ with $A_1$
in $S_2$) we see the isomorphism at the next step moving away from the root. 
As we move along, the isomorphic subtrees we extract in $A_i$ and $B_i$ will be
iterated truncations of $W_1$ alternating with iterated truncations of $W_2$.

Having obtained isomorphism at a new level using the inner subtrees, we
view them and the unique leftover edge continuing the path as located in the
outer subtrees in the other $(k-1)$-evine.  There we must truncate farthest
vertices to obtain the next level in the inner subtrees.  At each step, the
same operations are occurring in $A_i$ and in $B_i$, so we obtain the
isomorphism level by level.
\end{proof}

In our applications of Lemma~\ref{isom} to $T$, the labels do indicate subtrees
as in Definitions~\ref{trunketc} and~\ref{Sj}.
Although we know the $(k-1)$-vines $U_1$ and $U_2$ that appear in $T$ but are
not contained in $k$-vines, and we know the $(k-1)$-evines $S_1$ and $S_2$
that contain them, we don't yet know which of $S_1$ and $S_2$ contains which of
$U_1$ and $U_2$.  Postponing that issue, we show next that if we do know that
$U_i$ appears in $S_j$, so that $y_jz_j$ is the trunk edge of $U_i$ with $x_i$
occurring as $y_j$ or $z_j$, then we will be able to determine which of $y_j$
and $z_j$ is $x_i$.  This is equivalent to determining which of $C_j\cup Y_j$
and $D_j\cup Z_j$ is that occurrence of $U_i$.  If $U_i$ occurs as
$C_j\cup Y_j$, then $y_j$ is $x_i$ and $C_j$ is $A_i$ or $B_i$.  If $U_i$
occurs as $D_j\cup Z_j$, then $z_j$ is $x_i$ and $D_j$ is $A_i$ or $B_i$.

When we write that two $4$-tuples of rooted trees are equal, we mean that
corresponding entries are isomorphic.  Similarly, writing that a rooted tree
is not in a set of rooted trees means that it is not isomorphic to any of the
members, and the size of a set of rooted trees is the number of isomorphism
classes in it.

\begin{lemma}\label{onetrunk}
Fixing $i,j\in\{1,2\}$, if $S_j$ is known to be the maximal $(k-1)$-evine whose
central edge is the trunk edge of $U_i$, then the edge of $U_i$ serving as the
trunk edge is recognizable.
\end{lemma}
\begin{proof}
By definition, $x_i$ is $y_j$ or $z_j$.  If $x_i=y_j$, then $U_i=C_j\cup Y_j$
and the trunk edge of $U_i$ is in $B_i$.  If $x_i=z_j$, then $U_i=D_j\cup Z_j$
and the trunk edge of $U_i$ is in $A_i$.  Hence we want to determine which of
$y_j$ and $z_j$ is $x_i$.  

One of $\{C_j,D_j\}$ must be $A_i$ or $B_i$.  By symmetry, we may label the
subtrees so that $C_j\cong A_i$.  If $D_j\not\cong B_i$, then $y_j=x_i$.  Hence
we may assume both $C_j\cong A_i$ and $D_j\cong B_i$.

If $x_i=z_j$, then $Z_j\cong A_i$, but if $x_i=y_j$, then $Y_j\cong B_i$.
Hence if $Z_j\not\cong A_i$ or $Y_j\not\cong B_i$, then we have located $x_i$
in $S_j$.  Hence we may assume both $Z_j\cong A_i$ and $Y_j\cong B_i$.

We claim that this requires $A_i\cong B_i$.  Instead of building an isomorphism
as in Lemma~\ref{isom}, we apply Lemma~\ref{isom} by using another copy of
$S_j$.  Let $S'$ be the $(k-1)$-evine obtained by reversing the labeling of
$S_j$, making subtrees $C',Z',Y',D'$ in $S'$ isomorphic to $D_j,Y_j,Z_j,C_j$,
respectively.  Thus
$(C_j,Z_j,Y_j,D_j)=(A_i,A_i,B_i,B_i)$ and $(C',Z',Y',D')=(B_i,B_i,A_i,A_i)$.
With $S_j$ playing the role of $S_1$ and $S'$ playing the role of
$S_2$ in Lemma~\ref{isom}, we have the isomorphic pairs of subtrees
$(C_j,Y')$, $(Z_j,D')$, $(Y_j,C')$, and $(D_j,Z')$ needed to apply
Lemma~\ref{isom}.

We conclude that $C_j\cong Y_j$ or $D_j\cong Z_j$: that is, $A_i\cong B_i$.
Thus in this case $U_i$ is symmetric, and we may choose either edge incident to
$x_i$ along $P$ as the trunk edge of $U_i$.
\end{proof}

\begin{lemma}\label{easycase}
If some member of $\{C_1,C_2,D_1,D_2\}$ does not belong to
$\{A_1,A_2,B_1,B_2\}$ (as a rooted tree), then the trunk edges of $U_1$ and 
$U_2$ are recognizable.
\end{lemma}
\begin{proof}
By symmetry, we may assume $C_1\notin\{A_1,A_2,B_1,B_2\}$.
Thus $y_1\notin\{x_1,x_2\}$, so $z_1\in\{x_1,x_2\}$.  That is,
$D_1\cup Z_1\in\{U_1,U_2\}$, so
$D_1\in\{A_1,A_2,B_1,B_2\}$.  The rooted trees $A_1,A_2,B_1,B_2$ need not
be distinct; $D_1$ may be isomorphic to more than one of them.

If $D_1$ is isomorphic to exactly one of them, which by symmetry we may assume
is $A_1$, then $z_1=x_1$ and we know that the trunk edge in $U_1$ lies in
$B_1$, not $A_1$.  Also, since $D_1$ is not isomorphic to $A_2$ or $B_2$,
we have $U_2\esub S_2$, and Lemma~\ref{onetrunk} completes the proof.

Hence we may assume that $D_1$ is isomorphic to more than one of
$\{A_1,A_2,B_1,B_2\}$.  By symmetry, we again assume $D_1=A_1$.
If $D_1$ is not isomorphic to $A_2$ or $B_2$, then $D_1= A_1\cong B_1$ and it
does not matter which candidate is the trunk edge of $U_1$.  Since neither
$C_1$ nor $D_1$ is in $\{A_2,B_2\}$, we have $U_1\esub S_1$ and $U_2\esub S_2$
as actual subgraphs of $T$, and Lemma~\ref{onetrunk} on $U_2$ and $S_2$
completes the proof.

In the remaining case, by symmetry, $D_1=A_1\cong A_2$.
If also $B_1\cong B_2$, then $U_1\cong U_2$.  Now it does not matter which
$U_i$ we view as lying in which $S_j$ in $T$, so we can assign them arbitrarily
and use Lemma~\ref{onetrunk} to complete the proof.

Hence we may assume $D_1=A_1\cong A_2$ and $B_1\not\cong B_2$.  Again
$z_1\in\{x_1,x_2\}$, so $D_1\cup Z_1$ is $U_1$ or $U_2$.  Since
$B_1\not\cong B_2$, now $Z_1$ is isomorphic to just one of $B_1$ or $B_2$; let
it be $B_j$.  Now $z_1=x_j$, the trunk edge of $U_j$ is in $B_j$, and
$U_{3-j}\esub S_2$, so Lemma~\ref{onetrunk} completes the proof.
\end{proof}

In the remaining cases, $\{C_1,C_2,D_1,D_2\}\subseteq\{A_1,A_2,B_1,B_2\}$.
In light of Lemma~\ref{onetrunk}, our goal is to find which of $S_1$ and $S_2$
contains which of $U_1$ and $U_2$.  We break this into three lemmas, depending
on how many isomorphism classes of rooted trees comprise $\{A_1,A_2,B_1,B_2\}$.

\begin{lemma}\label{2pieces}
If $\C{\{A_1,A_2,B_1,B_2\}}\le 2$, then the trunk edges of $U_1$ and $U_2$
are recognizable.
\end{lemma}
\begin{proof}
If neither of $\{A_1,B_1\}$ is isomorphic to either of $\{A_2,B_2\}$, then
having at most two isomorphism classes requires $A_1\cong B_1$ and
$A_2\cong B_2$.  In this case, $U_1$ and $U_2$ are symmetric and in each
the choice of trunk edges from the two candidates does not matter.  This case
includes that of $A_1$, $A_2$, $B_1$, and $B_2$ being pairwise isomorphic.

In the remaining case, a member of $\{A_1,B_1\}$ is isomorphic to a member of
$\{A_2,B_2\}$.  By symmetry in the labeling, we may assume $A_1\cong A_2$ and
call this rooted tree $A$.  If $B_1\cong B_2$, then $U_1\cong U_2$, and we can
view $U_1$ and $U_2$ as contained in $S_1$ and $S_2$, respectively, applying
Lemma~\ref{onetrunk} to find the trunk edges.

Hence we may assume $B_1\not\cong B_2$.  Since the four subtrees lie in two
isomorphism classes, we may assume by symmetry in the indices that
$B_1\cong A_1\cong A_2\cong A$ and $B_2\not\cong A$.  Now $U_1$ is symmetric
and the choice of its trunk edge does not matter, so it remains to determine
the trunk edge of $U_2$.  If $B_2$ does not occur as any of the branches
$\{C_1,C_2,D_1,D_2\}$, then $B_2$ cannot contain the end of a longest path and
instead contains the trunk edge of $U_2$.

Let $B=B_2$.  In the remaining case, $B\in\{C_1,C_2,D_1,D_2\}$.  We have made
no comment yet distinguishing branches of $S_1$ or $S_2$, so by symmetry we may
assume $D_2\cong B$.  Now if $C_2\not\cong A$, then neither branch of $S_2$ is
$A_1$ or $B_1$, so $U_1\not\esub S_2$ in $T$.  This forces $U_2\esub S_2$, and
we apply Lemma~\ref{onetrunk}.  Hence we may assume $C_2\cong A$ and
$D_2\cong B$.  If $S_1\cong S_2$, then we can apply Lemma~\ref{onetrunk} with
$U_2\esub S_2$ or $U_2\esub S_1$.  Hence we may assume $S_1\not\cong S_2$.

Since $A_1\cong B_1\cong A\not\cong B\cong D_2$, we have $U_1\esub S_2$ only if
$Y_2\cong A$.  Therefore, if $Y_2\not\cong A$, then we know $U_1\nosub S_2$,
which forces $U_2\esub S_2$, and we apply Lemma~\ref{onetrunk}.

Hence we may assume $Y_2\cong A$ and $U_1\esub S_2$.  We are confused about how
to apply Lemma~\ref{onetrunk} only if $U_2$ occurs in both $S_1$ and $S_2$.
For $U_2\esub S_2$, we must have $Z_2\cong A_2\cong A$.  That is,
$C_2\cong Z_2\cong Y_2\cong A$ and $D_2\cong B$.  On the other hand, confusion
also requires having both $U_1$ and $U_2$ occur in $S_1$.  With
$S_1\not\cong S_2$, by symmetry in the labels we may assume
$C_1\cong Z_1\cong D_1\cong A$ and $Y_1\cong B$.  See Figure~\ref{S1S2}.

Now we modify Figure~\ref{S1S2} by relabeling to reverse $S_2$, viewing the
subtrees $(A,A,A,B)$ as $(D_2,Y_2,Z_2,C_2)$ instead of $(C_2,Z_2,Y_2,D_2)$.
With this relabeling, since $(C_1,Z_1,Y_1,D_1)=(A,A,B,A)$, we have 
$C_1\cong Y_2$, $Z_1\cong D_2$, $Y_1\cong C_2$, and $D_1\cong Z_2$.
By Lemma~\ref{isom}, we conclude $C_1\cong Y_1$ and $D_1\cong Z_1$.  In other
words, $A\cong B$, which is forbidden in this case.
\end{proof}
\begin{figure}[h]
\begin{center}
\includegraphics[scale=0.44]{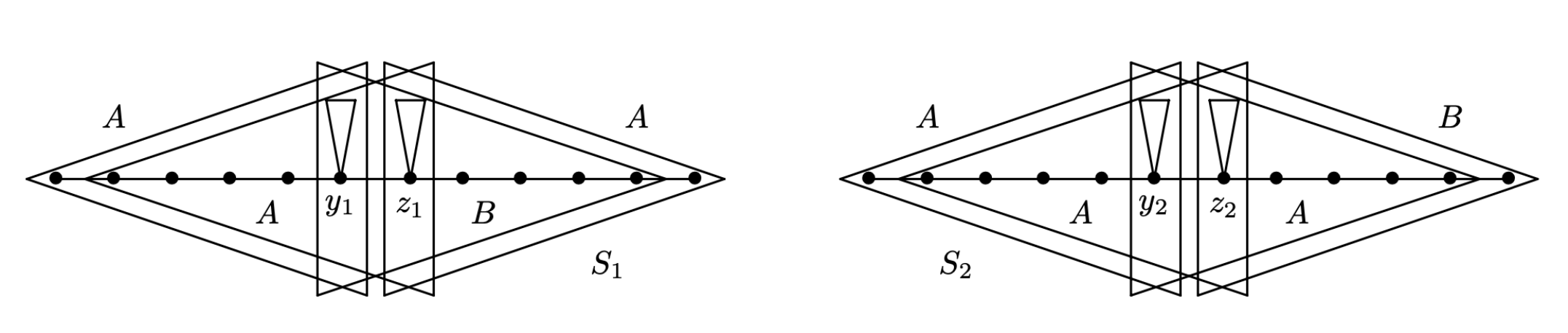}
\caption{$(k-1)$-evines in Lemma~\ref{2pieces}\label{S1S2}}
\end{center}
\end{figure}
\eject

\vspace{-.5pc}

\begin{lemma}\label{3pieces}
If $\C{\{A_1,A_2,B_1,B_2\}}= 3$, then the trunk edges of $U_1$ and $U_2$
are recognizable.
\end{lemma}
\begin{proof}
Again let $W_1=A_1\cap B_1$ and $W_2=A_2\cap B_2$.  By hypothesis,
two of $\{A_1,A_2,B_1,B_2\}$ are isomorphic.  By Lemma~\ref{onetrunk}, we can
recognize the trunk edges in $U_1$ and $U_2$ unless each of $U_1$ and $U_2$
appears in both $S_1$ and $S_2$ as a subgraph and $S_1\not\cong S_2$.

\smallskip
{\bf Case 1:} {\it $A_1\cong B_1$ or $A_2\cong B_2$, from the same $U_i$.}
By symmetry, we may assume $A_1\cong B_1\cong A$, so the choice of trunk edge
for $U_1$ does not matter.  As noted above, we must have $U_1$ in both
$S_1$ and $S_2$, so we have both $A\in\{C_1,D_1\}$ and $A\in\{C_2,D_2\}$.
However, since $A$ is neither $A_2$ nor $B_2$ and $U_2$ must also appear
in both $S_1$ and $S_2$, we may assume by symmetry in the labeling that
$C_1=C_2=A$ and that neither $D_1$ nor $D_2$ is $A$.

Since confusion requires $S_1\not\cong S_2$, we may assume $D_1=A_2$ and
$D_2=B_2$, with $U_2$ occurring as both $Z_1\cup D_1$ and $Z_2\cup D_2$.
In $S_1$, deleting the extremal vertices of $A$ (as $C_1$) and adding $W_2$
above the root yields $Z_1$ and hence $B_2$.
In $S_2$, deleting the extremal vertices of $A$ (as $C_2$) and adding $W_2$
above the root yields $Z_2$ and hence $A_2$.
Thus $A_2=B_2$, which contradicts the hypothesis of this case.

\smallskip
{\bf Case 2:} {\it $A_1$ or $B_1$ (from $U_1$) is isomorphic to $A_2$ or $B_2$
(from $U_2$).}  
By symmetry, we may assume $A_1\cong A_2\cong A$.  Here all of $\{A,B_1,B_2\}$
are different, but $W_1\cong W_2$ since each is obtained from $A$ in the same
way.

If $A\notin\{C_1,C_2,D_1,D_2\}$, then the ends of $U_1$ and $U_2$ occupied by
$A$ do not occur at an end of $S_1$ or $S_2$, so the trunk edges in $U_1$ and
$U_2$ both lie in $A$.  Hence we may assume $C_1\cong A$.  If
also $D_1\cong A$, then $U_1$ and $U_2$ both appearing in $S_1$ requires
$\{Y_1,Z_1\}=\{B_1,B_2\}$.  Also, since $C_1\cong A\cong D_1$, the rooted
subtrees at $y_1$ and $z_1$ outside the major pieces both equal $W_1$.  Hence
both $Y_1$ and $Z_1$ arise by deleting the extremal vertices of $A$ and adding
$W_1$ at the top.  This yields $Y_1\cong Z_1$ and $B_1\cong B_2$,
which contradicts the hypothesis.

Hence we may assume $D_1\not\cong A$.  By Lemma~\ref{easycase} and symmetry in
the indices, we may now assume $D_1\cong B_2$.  Since $U_1$ and $U_2$ must both
appear in $S_1$, we have $Y_1\cong B_1$ and $Z_1\cong A$; that is,
$U_1\cong C_1\cup Y_1$ and $U_2\cong Z_1\cup D_1$.

Now consider $S_2$.  By reversing the labeling of $S_2$ if needed, we may
assume $U_1=C_2\cup Y_2$ and $U_2=Z_2\cup D_2$.
Hence $C_2\in\{A,B_1\}$ and $D_2\in\{A,B_2\}$.  This leaves four cases.

{\bf(a)} $(C_2,D_2)=(A,B_2)$.  Here $S_2\cong S_1$, so
Lemma~\ref{onetrunk} applies with either assignment.

{\bf(b)} $(C_2,D_2)=(A,A)$.  Here the argument in Case 1 about deleting the
extremal vertices from $A$ and adding another copy of the root offshoots
above the root yields $Y_2\cong Z_2$ and hence $B_1\cong B_2$, a contradiction.

{\bf(c)} $(C_2,D_2)=(B_1,B_2)$.  In $S_1$, deleting the extremal vertices of
$B_2$ and adding $W_1$ above the root yields $B_1$.  In $S_2$, deleting the
extremal vertices of $B_2$ and adding $W_1$ above the root yields $A$.
Hence $B_1\cong A$, a contradiction.

{\bf(d)} $(C_2,D_2)=(B_1,A)$.  Now $(C_1,Z_1,Y_1,D_1)=(A,A,B_1,B_2)$, and also
$(C_2,Z_2,Y_2,D_2)=(B_1,B_2,A,A)$.  In particular, we have isomorphic
pairs $(C_1,Y_2)$, $(Z_1,D_2)$, $(Y_1,C_2)$, and $(D_1,Z_2)$.  By
Lemma~\ref{isom}, $A\in\{B_1,B_2\}$, again a contradiction.
\end{proof}

\begin{lemma}\label{4pieces}
If $\C{\{A_1,A_2,B_1,B_2\}}= 4$, then the trunk edges of $U_1$ and $U_2$
are recognizable.
\end{lemma}
\begin{proof}
Here $A_1,A_2,B_1,B_2$ are distinct.  Again by Lemma~\ref{onetrunk} we have the
desired result unless $S_1\not\cong S_2$ and each of $U_1$ and $U_2$ appears in
both $S_1$ and $S_2$.  By Lemma~\ref{easycase}, we may assume
$\{C_1,C_2,D_1,D_2\}\esub\{A_1,A_2,B_1,B_2\}$.

\medskip
{\bf Case 1:} {\it Some member of $\{A_1,A_2,B_1,B_2\}$ is not in
$\{C_1,C_2,D_1,D_2\}$.}
By symmetry, we may assume $A_1\notin\{C_1,C_2,D_1,D_2\}$.  In order to have
$U_1$ appear in both $S_1$ and $S_2$, we must have $B_1$ as an end-subtree in
both $S_1$ and $S_2$.  By symmetry, we may assume $D_1\cong D_2\cong B_1$.
Now, since $U_2$ appears in both $S_1$ and $S_2$ and $S_1\not\cong S_2$, we may
assume $C_1\cong A_2$ and $C_2\cong B_2$.  To complete the copies of $U_2$ in
$S_1$ and $S_2$, we must have $Y_1\cong B_2$ and $Y_2\cong A_2$.

At this point, the subtree $W_2$ that is $A_2\cap B_2$ rooted at $x_2$ appears
both in $S_1$ rooted at $y_1$ and in $S_2$ rooted at $y_2$.  These are the
portions of $Y_1$ and $Y_2$ that do not lie in $D_1$ and $D_2$, respectively.
Since $D_1\cong D_2\cong B_1$, deleting the extremal vertices of $B_1$ and
placing $W_2$ atop the root of $B_1$ yields $B_2$ by looking at $Y_1$ in $S_1$
and $A_2$ by looking at $Y_2$ in $S_2$.  Hence $A_2\cong B_2$, contradicting
the hypothesis.

%

\medskip
{\bf Case 2:} {\it All of $\{A_1,B_1,A_2,B_2\}$ appear in
$\{C_1,C_2,D_1,D_2\}$}, and hence there is a bijection via isomorphisms from
one set of subtrees to the other.  Recall that $U_i=A_i\cup B_i$.

Since each of $U_1$ and $U_2$ appears in both $S_1$ and $S_2$, we may assume by
symmetry that $C_1\cong A_1$, and hence $Y_1\cong B_1$ with $W_1$ rooted at
$y_1$.  Reversing the labeling of $U_2$ if necessary, this leaves
$D_1\cong B_2$ and $Z_1\cong A_2$, so $(C_1,Z_1,Y_1,D_1)=(A_1,A_2,B_1,B_2)$.
Since $U_1$ and $U_2$ must also appear in $S_2$, with
$\{C_2,D_2\}=\{A_2,B_1\}$ by the hypothesis of this case, we may reverse the
labeling of $S_2$ if necessary to obtain $(C_2,Z_2,Y_2,D_2)=(B_1,B_2,A_1,A_2)$
(see Figure~\ref{s1s2fig}).  By Lemma~\ref{isom}, we obtain $A_1\cong A_2$
or $B_1\cong B_2$, contradicting the hypothesis.
\end{proof}


\begin{corollary}\label{trunks}
The trunk edges of $U_1$ and $U_2$ are recognizable.
\end{corollary}
\begin{proof}
With $U_1,U_2,S_1,S_2$ and their subtrees $A_1,B_1,A_2,B_2$ and
$C_1,D_1,C_2,D_2$ known from the deck, Lemmas~\ref{easycase}
through~\ref{4pieces} treat an exhaustive set of cases
where arguments applying Lemmas~\ref{isom} and~\ref{onetrunk} are used
to recognize the trunk edges of $U_1$ and $U_2$,
thereby determining which of $\{A_1,B_1\}$ and which of $\{A_2,B_2\}$
contain the ends of a longest path $P$ in $T$.
\end{proof}

\section{Large Diameter and no Spi-center}

In this section we continue and complete the reconstruction of trees having
large diameter and no spi-center (no copy of $\Sp{\ell+1}$).  Henceforth we
assume $n\ge6\ell+11$.  With also $r\ge n-3\ell$, the usual computation from
Lemma~\ref{largek} yields $k\ge\ell+4$.  By Lemma~\ref{spi}, we can then count
the spi-centers in $T$ and thus recognize $\Sp{\ell+1}\nosub T$.

\begin{definition}
In the previous section we determined the two maximal $(k-1)$-vines $U_1$
and $U_2$ whose centers are distance $k-1$ from the endpoints of any longest
path in the $n$-vertex tree $T$ whose $(n-\ell)$-deck we are given.  Since
$k-1\ge\ell+1$ and $\Sp{\ell+1}\nosub T$, in both $U_1$ and $U_2$ only two
pieces have the full length $k-1$ (the ``major'' pieces).  Also, we have
successfully oriented each of these $(k-1)$-vines, meaning that we know which
is the ``outer'' major piece containing an endpoint of each longest path and
which is the ``inner'' major piece extending toward the other $(k-1)$-vine.

We maintain the notation $U_1$ and $U_2$ for those maximal $(k-1)$-vines, with
$x_i$ the center of $U_i$.  Since $\Sp{\ell+1}\nosub T$ and $k-1\ge \ell+1$,
every $r$-path in $T$ contains the path from $x_1$ to $x_2$; we use this
implicitly without mention throughout this section.  The distance between $x_1$
and $x_2$ is $r+1-2k$.  Since $k\ge(r-\ell-4)/2$ by Lemma~\ref{largek}, the
distance between $x_1$ and $x_2$ is at most $\ell+5$.

For $i\in\{1,2\}$, {\it inner extremal vertices} in $U_i$ are vertices at
distance $k-1$ from $x_i$ in the inner piece.  If there are multiple inner
extremal vertices in $U_i$, then let $x_i'$ be the vertex closest to $x_i$ in
the inner piece at which paths to inner extremal vertices in $U_i$ diverge.
If there is no such vertex, then let $x_i'$ be the unique extremal vertex in
the inner piece, having distance $k-1$ from $x_i$ along $P$.
\end{definition}

\begin{lemma}\label{offshoots}
For $i\in\{1,2\}$, all offshoots from $P$ in $T$ at vertices between $x_i$ and
$x_i'$ (not including $x_i'$) are seen in full in $U_i$.  Furthermore,
$x_1'$ and $x_2'$ are strictly between $x_1$ and $x_2$ on $P$, and we know
their order along $P$.
\end{lemma}
\begin{proof}
All vertices having distance at most $k-1$ from $x_i$ are seen in $U_i$.
An offshoot at a vertex between $x_i$ and $x_i'$ containing a vertex with
distance at least $k-1$ from $x_i$ would contradict the definition of $x_i'$.

For each $i$, the vertex $x_i'$ lies in the inner piece of $U_i$.  Since
we have determined which is the inner piece, we know $x_i'$ in $U_i$, and
the path from $x_i$ to $x_i'$ lies along $P$ in the direction toward $x_{3-i}$.
We also know the distance from $x_i$ to $x_i'$, and the distance from $x_1$ to
$x_2$ is $r+1-2k$.  Hence we know the order of $\{x_1,x_1',x_2',x_2\}$ along
$P$ and whether $x_1'=x_2'$.
\end{proof}

We consider cases depending on the order and distinctness of $x_1'$ and $x_2'$,
which we have just shown are recognizable.

\begin{lemma}\label{abab}
If $x_1,x_1',x_2',x_2$ occur along $P$ in the order $(x_1,x_2',x_1',x_2)$ with
$x_2'\ne x_1'$, then $T$ is $\ell$-reconstructible.
\end{lemma}
\begin{proof}
We know the outer pieces of $U_1$ and $U_2$ and the offshoots at $x_1$ and
$x_2$ with length less than $k-1$.  Adding the centers $x_1$ and $x_2$ and a
path of length $r+1-2k$ joining them yields a subtree of $T$ containing an
$r$-path we call $P$.  It remains to determine the offshoots from $P$ along the
path joining $x_1$ and $x_2$.  Every offshoot at a vertex between $x_i$ and
$x_i'$ is seen in full in $U_i$, by Lemma~\ref{offshoots}.  With the vertices
in the specified order, all the needed offshoots are seen in full.
\end{proof}

\begin{lemma}\label{x_1'=x_2'}
If $x_1'=x_2'$, then $T$ is $\ell$-reconstructible.
\end{lemma}
\begin{proof}
By Lemma~\ref{offshoots}, we recognize this case and we know everything in $T$
except the offshoots from $P$ at the vertex called both $x_1'$ and $x_2'$.
Their length is less than $k-1$, since $k-1\ge\ell+1$ and
$\Sp{\ell+1}\nosub T$.  Hence they appear in full in the $k$-vine centered at
$x_1'$.

We know the maximal $k$-vines in $T$; the $k$-centers are on $P$ between $x_1$
and $x_2$.  Each has only two offshoots of length $k$ from the center; we take
its rooted subgraph obtained by deleting those two pieces.  Knowing all of $T$
except the offshoots from $x_1'$, we can exclude from this list of subgraphs
the ones that come from $k$-centers other than $x_1'$, with correct
multiplicity.  What remains is the rooted subgraph at $x_1'$, completing the
reconstruction of $T$.
\end{proof}

\begin{lemma}\label{n-2l}
If $r\ge n-2\ell$ and $n\ge6\ell+11$ and $x_1'\ne x_2'$, then
$x_1,x_1',x_2',x_2$ must occur in the order $(x_1,x_2',x_1',x_2)$ along $P$.
\end{lemma}
\begin{proof}
For each $i$, the path from $x_i'$ to an extremal vertex in $U_i$ outside $P$
has length at most $\ell$, since $\Sp{\ell+1}\nosub T$.  Extending such a path
from the extremal vertex back through $x_i$ to the end of $P$ yields a path
with $2k-1$ vertices.  The resulting paths for $i\in\{1,2\}$ are disjoint when
the vertices are in the order $x_1,x_1',x_2',x_2$.  Since together these paths
have at most $2\ell$ vertices outside $P$, we have $r\ge 4k-2-2\ell$.
Lemma~\ref{largek} then yields $r\ge 2r-2\ell-8-2-2\ell$, which simplifies to
$r\le 4\ell+10$.  With $r\ge n-2\ell\ge 4\ell+11$, we have a contradiction.
\end{proof}

By Lemmas~\ref{offshoots}--\ref{n-2l}, we may assume in the
remainder of this section that $x_1,x_1',x_2',x_2$ are distinct and occur
in that order and that $n-3\ell\le r\le n-2\ell-1$,

\begin{lemma}\label{equidist}
If some longest path has vertices $y$ and $y'$ equidistant from the ends
(and distance at least $k$ from the ends) such that the offshoots from the path
at $y$ and $y'$ together have more than $\ell$ vertices, then $T$ is
$\ell$-reconstructible.
\end{lemma}
\begin{proof}
First we observe that this case is recognizable.  We consider subtrees
containing an $r$-vertex path and offshoots from the path only at two vertices
that are equidistant (with distance at least $k$) from the ends.  If in some
such subtree the offshoots have $\ell+1$ vertices in total, then the subtree
fits in a card (since $r\le n-2\ell-1$), and we see it in a csc.


Since $\Sp{\ell+1}\nosub T$ and $k-1\ge\ell+1$, every longest path contains the
path from $x_1$ to $x_2$, and thus $y$ and $y'$ lie along this path and are
$k$-centers in $T$.

Let $C_1$ be a largest subtree containing an $r$-path such that the
vertices along it with the same distance from the center as $y$ and $y'$
(these in fact are $y$ and $y'$ themselves) have degree $2$.  Since the
offshoots at $y$ and $y'$ have at least $\ell+1$ vertices, $C_1$ fits into a
card, we see it, and it gives us in full all offshoots from the path other than
those at $y$ and $y'$.

The exclusion argument with maximal $k$-vines used in Lemma~\ref{x_1'=x_2'} now
tells us the offshoots at $y$ and the offshoots at $y'$ but not which subgraph
is attached to which vertex on $P$.  If $C_1$ is symmetric (under reversal of
the $r$-path), then it does not matter which is which, so we may assume that
$C_1$ is not symmetric.  Let $Q$ and $Q'$ denote the two sets of offshoots yet
to be assigned to $y$ or $y'$.  (If $y=y'$, at the center of an $r$-path with
$r$ being odd, then the exclusion argument for the maximal $k$-vines gives the
offshoots from the center, without needing to distinguish $Q$ from $Q'$; hence
we may assume $y\ne y'$.)

Since the distance between $x_1$ and $x_2$ is at most $\ell+5$, and in this
section we know $k-1\ge\ell+3$, both $U_1$ and $U_2$ contain the center of $P$.
Since we also know the distance between $x_i$ and the center, and $C_1$ shows
us which path in the inner piece of $U_i$ lies along $P$, from $C_1$ we can
compute the number $\alpha_i$ of vertices in $U_i$ that lie in offshoots from
$P$ at vertices between $x_i$ and the center.  Since $C_1$ in fact shows the
offshoots from $P$ at all vertices other than $y$ or $y'$, when
$\C{V(Q)}\ne\C{V(Q')}$ we can tell which of $Q$ and $Q'$ should be associated
with the member of $\{y,y'\}$ closer to $x_1$ to reach a total of $\alpha_1$
vertices on that side.

Hence we may assume $\C{V(Q)}=\C{V(Q')}$.  We may also assume
$\alpha_1\le\alpha_2$, by symmetry.  If $\alpha_1=\alpha_2$, then we may
assume by symmetry that $U_1$ and $U_2$ are indexed so that the outside
piece of $U_2$ is at least as big as the outside piece of $U_1$.  Now let
$C_2$ be a largest long subtree having offshoots on one side of the center 
that agree with $U_1$ (as seen in $C_1$), allowing also offshoots at the vertex
whose distance from the end agrees with $y$.  Since $C_2$ omits at least half
of the vertices outside $P$, and $n-r\ge 2\ell$, we have deleted enough so that
$C_2$ fits in a card.  We see $C_2$, and it tells us which of $Q$ and $Q'$ is
attached to the member of $\{y,y'\}$ in $U_1$ (this uses that $C_2$ is
asymmetric).

If $\alpha_1<\alpha_2$, then instead we define $C_2$ having an $r$-path
and offshoots at the member of $\{y,y'\}$ in $U_1$ but using only $\alpha_1+1$
of the vertices in offshoots from $P$ at vertices on the side of the center in
$U_2$.  This distinguishes $y$ and $y'$ in $C_2$, which is again small enough
to fit into a card and tell us which of $Q$ and $Q'$ is attached to $y$.
\end{proof}

\begin{theorem}\label{nospi}
If $r\ge n-3\ell$ and $n\ge 6\ell+11$, and $\Sp{\ell+1}\nosub T$,
then $T$ is $\ell$-reconstructible.
\end{theorem}
\begin{proof}
As we have noted, the remaining case is $n-3\ell\le r\le n-2\ell-1$ with the
vertices $x_1,x_1',x_2',x_2$ distinct and occurring in that order along $P$.
Let $d(u,u')$ denote the distance between vertices $u$ and $u'$.  Since we know
$U_i$, we know $d(x_i,x_i')$.  By symmetry, we may assume
$d(x_1,x_1')\le d(x_2,x_2')$.

For $v\in V(P)$, let $f(v)$ be the number of vertices in the offshoots from $P$
at $v$.  We next prove $f(x_1')+f(x_2')\ge\ell+1$.  There are $r-2k$ vertices
between $x_1$ and $x_2$ on $P$.  The offshoots at $x_1'$ and $x_2'$ contain
extremal vertices in $U_1$ and $U_2$.  The paths from $x_1$ and $x_2$ to
extremal vertices have length $k-1$.  Hence the offshoots at $x_1'$ and $x_2'$
together contain at least $2k-2-(r-2k)$ vertices, due to the order
$(x_1,x_1',x_2',x_2)$.  By
Lemma~\ref{largek} and the case $r\ge n-3\ell$ and $n\ge 6\ell+11$,
$$4k-r-2\ge 2r-2\ell-8-r-2\ge r-2\ell-10\ge n-5\ell-10\ge \ell+1.$$
This completes the proof that $f(x_1')+f(x_2')\ge\ell+1$.

\smallskip
{\bf Case 1:} {\it $d(x_1,x_1')=d(x_2,x_2')$.}
In this case $x_1'$ and $x_2'$ are equidistant from the center.  Since we have
proved $f(x_1')+f(x_2')\ge\ell+1$, Lemma~\ref{equidist} implies that $T$ is
$\ell$-reconstructible.

\smallskip
{\bf Case 2:} {\it $d(x_1,x_1')<d(x_2,x_2')$, by symmetry.}
Let $x_1''$ be the vertex at distance $d(x_1,x_1')$ from $x_2$ along the path
from $x_2$ to $x_1$.  Similarly, let $x_2''$ be the vertex at distance
$d(x_2,x_2')$ from $x_1$ along the path from $x_1$ to $x_2$.  The order of the
six vertices along $P$ is $(x_1,x_1',x_2',x_2'',x_1'',x_2)$ or
$(x_1,x_1',x_2'',x_2',x_1'',x_2)$, depending on whether
$d(x_2,x_2')>d(x_1,x_2)/2$.

If $f(x_1')+f(x_1'')$ or $f(x_2')+f(x_2'')$ exceeds $\ell$, then
Lemma~\ref{equidist} completes the proof.  Hence we may assume
$f(x_1')+f(x_1'')\leq \ell$ and
$f(x_2')+f(x_2'')\leq \ell$.  On the other hand, we have proved
$f(x_1')+f(x_2')\ge \ell+1$.
         
We know the locations of $x_1,x_1',x_1'',x_2,x_2',x_2''$ along $P$.
Since $f(x_1')+f(x_2')\ge\ell+1$, omitting the offshoots at the four
vertices $x_1',x_1'',x_2',x_2''$ leaves a long subtree that fits in a card.
Hence we see a largest long csc $C_1$ in which the vertices at these specified
distances along the $r$-path have degree $2$; this csc shows us all offshoots
at the other vertices.

Since $f(x_1')+f(x_1'')\le \ell$ and $r\le n-2\ell$, a largest long subtree
having offshoots from the $r$-path only at the vertices having the distances of
$x_1'$ and $x_1''$ from the ends of the path fits in a card and shows us the
offshoots at $x_1'$ and $x_1''$, but not which is at which vertex.  Similarly,
we obtain the offshoots at $x_2'$ and $x_2''$ but do not know which set is 
attached to which vertex.

If $d(x_2,x_2')>d(x_2,x_1)/2$, then the order along $P$ is
$(x_1,x_1',x_2',x_2'',x_1'',x_2)$.  In this case the offshoots at $x_2''$ and
$x_1''$ appear in $U_2$, and we know them because we know all of $U_2$ and
which of its major pieces is the inner one.  This distinguishes which set of
offshoots in each pair $(x_1',x_1'')$ and $(x_2',x_2'')$ is attached to which
vertex, completing the reconstruction of $T$.

Hence we may assume $d(x_2,x_2')\le d(x_2,x_1)/2$, and the six vertices occur
in the order $(x_1,x_1',x_2'',x_2',x_1'',x_2)$ along $P$, including the
possibility $x_2''=x_2'$.  Again we can use $U_2$ to see the offshoots at
$x_1''$ and assign offshoots to the pair $(x_1',x_1'')$, but we still must
assign the offshoots at $x_2'$ and $x_2''$.  If $x_2'=x_2''$, then there is
only one set of offshoots and no need to distinguish them, so we may assume
$d(x_2,x_2')\le[d(x_2,x_1)-1]/2\le(\ell+4)/2$.

Some offshoot at $x_2'$ contains a path to an inner extremal vertex of $U_2$.
Hence $f(x_2')\ge k-1-d(x_2,x_2')\ge k-3-\ell/2$.  Since $k\ge\ell+4$ (using
$n\ge6\ell+11$), we have $f(x_2')\ge \ell/2+1$.
Since $f(x_2')+f(x_2'')\le \ell$, we conclude $f(x_2')>f(x_2'')$, which allows
us to assign the offshoots to $x_2'$ and $x_2''$ correctly, completing the
reconstruction of $T$.
\end{proof}

\section{Small Diameter}

In this section we consider trees with small diameter, in particular
with $r\le n-3\ell-1$.  First we consider those having sparse cards and then
those not having sparse cards.  When $r\le n-3\ell-1$, any connected card has
at least $2\ell+1$ vertices outside any path.

We note first a case that we have already handled.

\begin{corollary}\label{s3spider}
If $n\ge6\ell+5$ and $r\le n-\ell$ and $T$ contains a sparse card
that is a $3$-legged spider, then $T$ is $\ell$-reconstructible.
\end{corollary}
\begin{proof}
A $3$-legged spider has at most $(3r-1)/2$ vertices.  Such a tree as a card
requires $(3r-1)/2\ge n-\ell\ge 5\ell+5$, so $r\ge(10\ell+11)/3$.  In
particular, $r\ge 3\ell+6$.  By Lemma~\ref{spider}, $T$ has exactly one
spi-center.  Now Lemma~\ref{onelctr} applies and $T$ is $\ell$-reconstructible.
\end{proof}

\vspace{-1pc}

Henceforth in this section we may assume that $T$ has no sparse card that
is a $3$-legged spider.
Since all longest paths in a tree have the same center, the central vertices of 
all long cards are the same.  An {\it optimal} sparse card is a sparse card
such that, among all sparse cards, the primary vertex is closest to the center.



\begin{lemma}\label{optcard}
Suppose $n\ge6\ell+\c$ and $r\le n-3\ell-1$.
Let $C$ with primary vertex $z$ be an optimal sparse card, having an
$r$-path $\la\VEC v1r\ra$ with $z=v_j$ and $j\le(r+1)/2$.
Every $r$-path in $T$ contains the path from $z$ to the center of $T$,
and all optimal sparse cards have the same primary vertex $z$ in $T$.
\end{lemma}
\begin{proof}
Since all longest paths in a tree contain the center, when $j=\FL{(r+1)/2}$
there is nothing to prove, so we may assume $j<\FL{(r+1)/2}$.
Let $P=\la \VEC v1r\ra$.

Let $y$ be the central vertex of $T$ closest to $z$.  Let $T'$ be the component
of $T-z$ containing $y$.  Let $P'$ be another $r$-path in $T$; like every
$r$-path, $P'$ contains $y$.  If $P'$ does not contain the path from $y$ to $z$,
then $P'$ leaves the path from $y$ to $z$ at some vertex $w$ before $z$.
In the direction away from $z$, since $P'$ is a longest path, it extends
$\FL{r/2}$ vertices beyond $y$ in $T'$, just as $P$ does.  We now form a sparse
card $C'$ with $r$-path $P'$ and primary vertex $w$.  We keep $n-\ell-r$
vertices from the component of $C-w$ containing $z$.  The number of vertices in
that component is greater than $n-\ell-r$, since it contains $z$ and the
offshoots from $P$ at $z$ in $C$, which total $n-\ell-r$ vertices.  Hence $C'$
exists.  Since $w$ is closer to the center of $T$ than $z$ is, $C'$ contradicts
the optimality of $C$.


Suppose that distinct vertices $z$ and $z'$ are primary vertices in optimal
sparse cards $C$ and $C'$, respectively.  Since $z$ and $z'$ lie on every
$r$-path, the path $R$ joining them is in both $C$ and $C'$.  If $z$ or $z'$
lies outside the central $r-2j+2$ vertices of the designated $r$-path in the
other card, then $T$ has a path with more than $r$ vertices.  Otherwise,
the offshoots from $R$ at the primary vertex in $C$ are offshoots from $R$ at a
non-primary vertex in $C'$, and vice versa.  Since $C$ and $C'$ are cards, this
yields $n-\ell-r+j-1\le \ell+j-1$, which contradicts $r\le n-3\ell-1$.
\end{proof}

\vspace{-1pc}

\begin{lemma}\label{sparse3}
Suppose $n\ge6\ell+\c$ and $r\le n-3\ell-1$.  If $T$ has a sparse card of degree
$3$, then some optimal sparse card has degree $3$.  If no sparse
card has degree $3$, then every sparse card is optimal.
\end{lemma}
\begin{proof}
Let $C$ be any sparse card with primary vertex $z$ on the $r$-path $P$.  Let
$P'$ be an $r$-path in another sparse card.  Since $r\le n-3\ell-1$, in any
sparse card the offshoots from the $r$-path at the primary vertex have at least
$2\ell+1$ vertices, and the card omits only $\ell$ vertices.

If $P'$ contains the path from $z$ to the central vertex $y$, then there are at
most $\ell$ vertices available for offshoots from $P'$ at any vertex of
$P'$ closer to $y$ than $z$, so such a card could not have a primary vertex
closer to $y$.  If $P'$ departs from $P$ at some vertex $w$ before $z$ in
traveling from $y$, then the sparse card $C'$ formed in the proof of 
Lemma~\ref{optcard} has primary vertex $w$.  That card has degree $3$,
since $w$ has degree $2$ in $C$.

That is, if $C$ is not an optimal sparse card, then we obtain a sparse
card of degree $3$ whose primary vertex is closer to the center.
This implies both conclusions claimed.
\end{proof}

\vspace{-1pc}

\begin{definition}\label{longdef}
{\it Peripheral vertices} are endpoints of $r$-paths in $T$.
When $T$ has a sparse card, let $z$ denote the unique vertex of $T$ that by
Lemma~\ref{optcard} is the primary vertex in every optimal sparse card.
The {\it primary value} of $T$ is the index $j$ with $j\le(r+1)/2$
such that $z$ is $v_j$ on an $r$-path $P$ indexed as $\la\VEC v1r\ra$.

For a given optimal sparse card
$C$, let $W$ denote the union of the offshoots from $P$ at $z$ in $C$, with
$W^*$ being the union of the components of $T-z$ containing the offshoots in
$W$.  Let $T_1$ and $T_2$ be the components of $T-z$ containing $v_{j-1}$ and
$v_{j+1}$, respectively, where $j$ is the primary value.  Although
we see $W$ in $C$, from $C$ we do not know any of $W^*$, $T_1$, or $T_2$.

When $T$ has a sparse card and primary value $j$, a {\it long-legged card} is a
connected card $C'$ having a leg of length more than $r-j$.  When
$r\le n-3\ell-1$, a long-legged card has a branch vertex, since otherwise the
card is a path, requiring $r\ge n-\ell$.  See Figure~\ref{ccpfig}.
\end{definition}

\vspace{-1pc}

\begin{figure}[h]
\begin{center}
\includegraphics[scale=0.5]{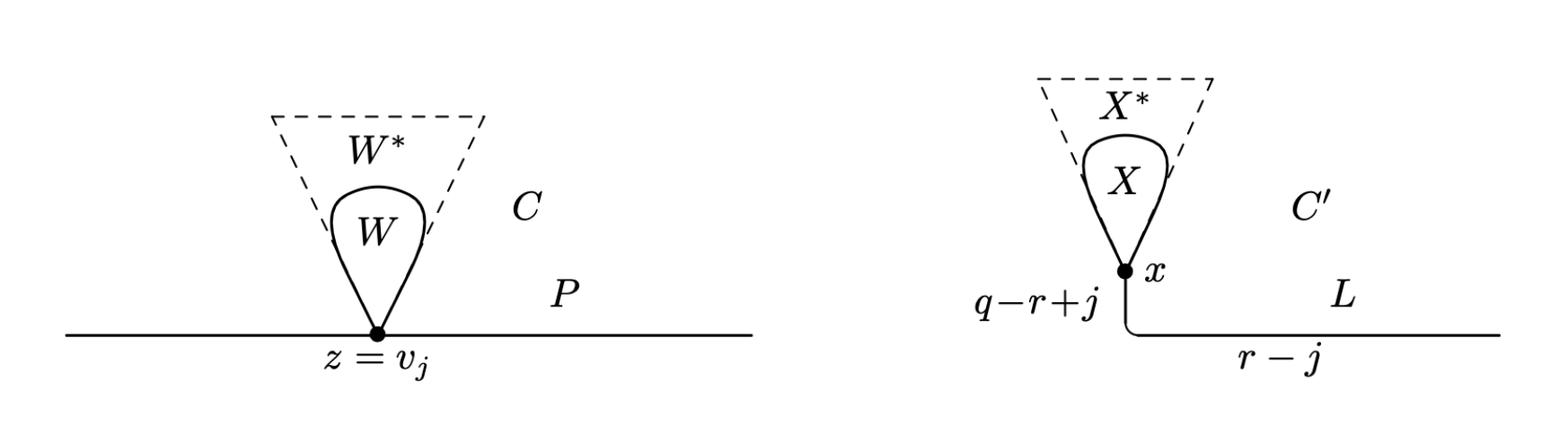}
\caption{Sparse card $C$ and long-legged card $C'$\label{ccpfig}}
\end{center}
\end{figure}

\begin{lemma}\label{longleg}
Suppose $n\ge6\ell+\c$ and $r\le n-3\ell-1$.  If $T$ has a sparse card of degree
$3$, and $T$ has a long-legged card, then $T$ is $\ell$-reconstructible.
\end{lemma}
\begin{proof}
By Lemma~\ref{sparse3}, $T$ has an optimal sparse card $C$ of degree $3$.
Relative to $C$, we use notation $z,j,P,T_1,T_2,W,W^*$ as in
Definition~\ref{longdef}.
Since $C$ has degree $3$, $W$ is a single offshoot from $P$.  Since $C$ is a
card and $r<n-3\ell$, in $W$ there are at least $2\ell+1$ vertices.
Let $q$ be the maximum length of the legs in long-legged cards, so $q>r-j$.
Let $\cC'$ be the family of long-legged cards having a leg of length $q$.
For a fixed card $C'\in\cC'$, let $L$ be the leg of length $q$, ending at
branch vertex $x$, and let $X=C'-V(L-x)$, as shown in Figure~\ref{ccpfig}.
Let $X^*$ be the component of $T-V(L-x)$ containing $x$.

\medskip
\noindent
{\bf Step 1:} {\it The primary vertex $z$ in $C$ is on $L$ in $C'$, at distance
$r-j$ from the leaf of $L$, and $X\esub W^*$.  Also, $x$
and $X^*$ are independent of the choice of $C'$ in $\cC'$, and $j-1\le\ell$.}

We first prove $z\in V(L)$.  Suppose not.  Let $z'$ be the vertex of
$C'$ closest to $z$ in $T$, and let $R$ be the path from $z$ to $z'$.
Since no path with length more than $r-j$ starts at $z$, we cannot have
$z'\in V(X)$ (and hence also $z\notin V(X)$), and for the same reason $z'$
lies on $L$ within distance $r-j$ of the leaf in $L$.  Let $L'$ be the portion
of $L$ from $z'$ to the leaf end of $L$.

Since $\C{V(W)}\ge2\ell+1$ and $C'$ omits only $\ell$ vertices, $W$ intersects
$C'$.  Since $T$ has no cycles and $z$ has a neighbor in $W$, the neighbor of
$z$ on $R$ must be the neighbor of $z$ in $W$.  Hence $P$ lies completely
outside $C'$, since otherwise $T$ has paths from $z$ to $C'$ through $P$ and
through $W$.  Now $L'\cup R\cup\la \VEC vjr\ra$ is a path.  If $L'$ has length
at least $j$, then this path is longer than $P$.  If $L'$ has length less than
$j$, then replacing $L'$ with $R\cup\la\VEC vjr\ra$ and deleting some vertices
of $X$ yields a long-legged card having a longer leg than $C'$.  Both
possibilities are forbidden, so in fact $z=z'\in V(L-x)$.

We have already observed that $L'$, ending at $z$, has length at most $r-j$.
Since $z$ has degree $2$ in $C'$, the card $C'$ has vertices in two components
of $T-z$, and $X$ is contained in one of them.  Since $C'$ is a card, $X$ has
at least $2\ell+1$ vertices outside any path, but the only component of $T-z$
having more than $\ell$ vertices outside $P$ is $W^*$.  Hence $X\esub W^*$, and
the neighbor of $z$ on the path to $x$ is not in $P$.  If $L'$ has length less
than $r-j$, we can now replace $L'$ in $C'$ with a path of length $r-j$ from
$z$ in $P$ and delete vertices from $X$ to obtain a long-legged card with leg
longer than $C'$, which contradicts the choice of $C'$.  Hence $L'$ has length
exactly $r-j$.

Suppose that there are choices for $C'$ with two different branch vertices
playing the role of $x$ at the end of the leg of length $q$.  Both have 
distance $q-r+j$ from $z$, so neither of these vertices is on the path from $z$
to the other.  Hence the paths to them from $z$ diverge.  Both cards now
contain at least $2\ell+1$ vertices of $W^*$ outside a longest path that do not
appear in the other, which prevents the cards from having $n-\ell$ vertices.
Hence there is only one choice for $x$.  Since $x$ is independent of the choice
of $C'$ in $\cC'$, and $z$ has fixed distance $r-j$ from the end of $L$, the
subtree $X^*$ also is independent of the choice of $C'$.  Furthermore,
since $X\esub W^*$, also $X^*\esub W^*$.  This and $z$ having degree $2$
in $C'$ imply that the card $C'$ omits at least $j-1$ vertices of $P$,
so $j-1\le \ell$.

\medskip
\noindent
{\bf Step 2:} {\it
Let $\hC$ be a csc of $T$ having a leg $\hL$ of length at least $q$, let $\hx$
be the branch vertex of $\hC$ at the end of $\hL$, and let $\hX=\hC-V(\hL-\hx)$.
If $\hC$ has at least $\ell-j+2$ vertices in offshoots from its longest path,
then $\hX\esub X^*$.}

Let $\hq$ be the length of $\hL$.  Every longest path in $\hC$ contains $\hL$,
since $\hq\ge r-j$.  Let $\hx'$ be the vertex of $C'$ at which the path
from $\hx$ to $V(C')$ reaches $C'$; it may be in $X$ or in $L-x$.

First suppose $\hx'\in V(X)$.  In this case that the path $\hR$ from $\hx$ to
$x$ lies entirely in $X^*$ and uses no edge of $L$.  If $x\in\hX$, then both
$\hx$ and $x$ lie in $\hX$, so also $\hR$ lies within $\hX$ and uses no edge of
$\hL$.  Thus $\hR$ meets $L$ at $x$ and $\hL$ at $\hx$, and the union
$L\cup\hR\cup\hL$ is a path longer than $P$.  On the other hand, if
$x\notin\hX$, then $\hX\cap C'$ is contained in the component of $C'-x$ that
contains $\hx'$.  Now $\hX$ lies in the component of $T-x$ containing $\hx'$,
which is $X^*$, as desired.

The remainder of this step consists of showing that the other case
$\hx'\notin V(X)$ cannot occur.  Suppose $\hx'\notin V(X)$, which implies
$\hx'\in V(L-x)$.  We first prove $\hx'\in V(\hX)$.  Since $\hx\in V(\hX)$ and
since $\hx'$ is where $\hR$ reaches $C'$, the assumption $\hx'\notin V(\hX)$
puts all of $\hX$ outside $C'$, including the $\ell+2-j$ vertices of $\hC$
that have been assumed to lie in offshoots from a longest path in $\hC$.
Also, since $z\in V(L-x)$, the card $C'$ lacks at least $j-1$ vertices of $P$
from $T_1$ or $T_2$.  We now have at least $\ell+1$ vertices missing from the
card $C'$, a contradiction.
Thus $\hx'\in V(\hX)$.

We next prove $\hx\in V(C')$, which is equivalent to $\hx=\hx'$.  Suppose
$\hx\ne\hx'$.  Since $\hx',\hx\in\hX$, the path $\hR'$ joining $\hx'$ and $\hx$
lies entirely in $\hX$.  Since $\hL$ contains no vertex of $\hX$ other than
$\hx$, the assumption $\hx\ne\hx'$ implies that $\hL$ shares no vertices with
$L$.  Let $L''$ be the part of $L$ from $\hx'$ to the leaf in $C'$.  Since
$\hq\ge q$, replacing $L''$ with $\hR\cup\hL$ (and deleting some of $V(X)$)
yields a card with a longer leg than $C'$.  Hence $\hx=\hx'$, which yields
$\hx\in V(L)$.

Now we consider how $\hL$ departs from $\hx$, in three cases.
(1) If $\hL$ does not depart $\hx$ along $L$, then either $L''\cup\hL$ is a
path longer than $P$, or replacing $L''$ with $\hL$ contradicts the choice of
$C'$ (this uses $q\ge r-j$).
(2) If $\hL$ departs $\hx$ along $L''$, then replacing $L''$ with $\hL$ yields
a card with a longer leg than $C'$, since $\hq\geq q$.
(3) We may therefore assume that $\hL$ departs $\hx$ along $L$ toward $x$.
Since $L''\cup \hL$ cannot be longer than $P$, vertex $\hx$ must be farther
from $X$ than $z$.  In this case, the $\ell-j+2$ vertices of $\hC$ (in $\hX$)
in offshoots from its longest path are omitted from $C'$, along with at least
$j-1$ vertices from $P$ in $T_1$ or $T_2$.  This contradicts
that $C'$ omits exactly $\ell$ vertices.  The contradiction arose under
the assumption $\hx'\notin V(X)$, which means that only the earlier case
$\hx'\in V(X)$ occurs, where we proved $\hX\esub X^*$.

\medskip
\noindent
{\bf Step 3:} {\it Reconstruction of $T-V(X^*-x)$.}
Recall that $\cC'$ is the family of long-legged cards having a leg of length
$q$.  For any $C'\in\cC'$, we have $X\esub W^*$ by Step 1, and hence
$X^*\esub W^*$.  Also by Step 1, the branch vertex $x$ at the end of the leg is
the same in all such $C'$.  We have also observed that every such card $C'$
omits at least $j-1$ vertices from $P$, so $j-1\le \ell$.  Choose $C'\in\cC'$
to maximize the length of the longest offshoot from the leg $L$ at $x$; let
$t'$ be this maximum length.  We also know that $C'$ has at least $2\ell+1$
vertices outside its longest path.  Therefore, in $C'$ we can find a subtree
$U$ with a leg $L$ of length $q$ ending at branch vertex $u$ such that $U$
contains a path of length $t'$ from $u$ outside $L$ plus exactly $\ell+1$ other
vertices in the offshoots from $L$ at $u$.  Any such subtree $U$ is a particular
choice of $\hC$ in Step 2, so $u\in X^*$.  Furthermore, since $U$ has the
maximum possible length beyond $u$, we must have $u=x$, and the offshoots in
$U$ are contained in $X^*$ and $W^*$.

Now let $C_1$ be a largest subtree in $T$ that contains $U$ but attaches no
additional vertices of $X^*$ to any vertex of $U$.  By Step 2, any $C_1$ grown
from $U$ in this way must have $u=x$, and its $\ell+1$ vertices in offshoots
outside the specified path extending the leg by length $t'$ are contained in
$X^*$ and $W^*$.  Therefore, $C_1$ omits at least $\ell$ vertices from the set
of at least $2\ell+1$ such vertices in $X$.  Thus $C_1$ fits in a card.
Therefore in $C_1$ we see all the offshoots from $L-x$ in full,
giving us all of $T$ outside $X^*$.

\medskip
\noindent
{\bf Step 4:} {\it Reconstruction of $X^*$.}
It remains to determine the offshoots from $x$ that comprise $X^*$.
Among the cards in $\cC'$, which all have the same vertex $x$ at the end of the
long leg, choose $C'$ to minimize $d_{C'}(x)$, and among those choices choose
$C'$ to minimize the size of a smallest offshoot $Y$ in $C'$ from $L$ at $x$.
If some offshoot in $T$ at $x$ other than $Y$ appears only partially in $C'$,
then we can alter $C'$ to add a vertex from that offshoot and delete a vertex
from $Y$.  If this eliminates $Y$, then either the degree of the branch vertex
at the end of the leg decreases or we obtain a long-legged card with a longer
leg.  Hence in $C'$ we see in full all offshoots from $L$ in $T$ at $x$ that
have any vertices in $C'$, except possibly the one containing $Y$.
We must find the remaining offshoots from $x$ in $T$.  

Fixing this $C'$ and hence $X$, let $d$ be the degree of $x$ in $X$, and let
$\VEC Q1d$ be the offshoots from $x$ in $X$, in nonincreasing order of size
(here $Q_d=Y$).  Note $d\ge2$, since $x$ is a branch vertex in $C'$.
As observed above, the offshoots $\VEC Q1{d-1}$ from $x$ in $C'$ are offshoots
from $x$ in $T$, seen in full in $C'$.

Since $C'$ is a card, $\VEC Q1d$ together have at least $2\ell+1$ vertices
outside any path, so they have at least $2\ell+2$ vertices in total.  Let
$\cC_1$ be the family of subtrees having a leg of length $q$ and $d-1$
offshoots at the end of the leg, with $d-2$ of them being $\VEC Q1{d-2}$.
Every member of $\cC_1$ fits in a card, since otherwise there would be a choice
for $C'$ with $x$ having degree less than $d$ in $X$.

\smallskip
{\bf Case 1:} {\it $d\ge4$.}
Since $\VEC Q1{d-2}$ are the largest offshoots in $X$ and include at least
half of the offshoots, they together have at least $\ell+1$ vertices, and since
the longest path from $x$ in $X$ has length at most $j-2$, these offshoots have
at least $\ell-j+3$ vertices outside a longest path in $C'$.
The union of these offshoots with the leg satisfies the conditions for $\hC$ in
Step 2, in every subtree in $\cC_1$ the branch vertex at the end of the leg is
in $X$, and it must equal $x$ because moving farther into $X$ would not allow
having an offshoot as large as $Q_1$.

We have noted that every member of $\cC_1$ fits in a card.  A largest member
of $\cC_1$ shows us $Q_{d-1}$ or another offshoot of the same size as its
smallest offshoot.  We know the multiplicity of each such csc, since we know
$T_1$ and $T_2$ and hence the number of peripheral vertices in each.  Thus we
also know all the members of $\cC_1$ in which the smallest offshoot is obtained
by deleting vertices of $Q_{d-1}$ (or another offshoot of the same size).  After
excluding those members, a largest remaining member of $\cC_1$ shows us the
next largest offshoot at $x$.  Continuing the exclusion argument allows us to
find all the offshoots at $x$ (there may be more than $d$).

\smallskip
{\bf Case 2:} {\it $d=3$.}
In $C'$ we see $Q_1$ and $Q_2$; these are offshoots in $T$ seen in full in
$C'$.  Since $Q_1\cup Q_2\cup Q_3$ has at least $2\ell+2$ vertices and $Q_1$ is
no smaller than $Q_2$, together $Q_1$ and $Q_3$ have at least $\ell+1$
vertices.  By the minimality of $Y$, the extension $Q^*_3$ of $Q_3$ in $T$ is
no bigger than $Q_2$.  Recall that every member of $\cC_1$ fits in a card.
Since we know $Q_2$ and the number of peripheral vertices in $T_2$
(and in $T_1$ if $j=(r+1)/2$), we can exclude from $\cC_1$ all the members
whose second offshoot arises by deleting vertices of $Q_2$ (with multiplicity).
A largest remaining member shows us the full $Q^*_3$.  Continuing the exclusion
argument allows us to find the remaining offshoots at $x$.

\smallskip
{\bf Case 3:} {\it $d=2$.}
We know $Q_1$ as a full offshoot in $T$, but we do not see in full the
next-largest offshoot from $x$ in $T$ (unless it has the same size as $Q_1$ and
we have several choices for $C'$).  Let $Q_2^*$ be this offshoot in $T$.
A card $C_2\in\cC'$ having two offshoots
$Y_1$ and $Y_2$ as equal in size as possible shows us $Q_2^*$ as $Y_2$ if the
sizes of the two offshoots differ by more than $1$, after which we can use
exclusion to find any smaller offshoots.

Hence we may assume that the sizes of $Y_1$ and $Y_2$ in $C_2$ differ by at
most $1$.  In particular, $\C{V(Y_2)}=\FL{(n-\ell-q-1)/2}$.  With $q\ge r-j+1$
and $r\le n-3\ell-1$, we have $n-q\ge n-r+j-1\ge 3\ell+j$.  Hence
$$\C{V(Y_2)}=\FL{(n-\ell-q-1)/2}\ge\FL{(2\ell+j-1)/2}>\ell,$$
where the last inequality uses $j\ge2$.
Offshoots from $L$ at $x$ in $T$ that do not contain $Y_1$ or $Y_2$ are 
completely omitted from $C_2$.  Hence in total they have at most $\ell$
vertices and are smaller than $Y_2$.

Because $X^*$ is an offshoot from a vertex in an offshoot from $z$, it
has length at most $j-2$.  Since in Step 1 we showed $j-1\le \ell$,
the length of $X^*$ is less than $\ell$.  Since $Y_1$ and $Y_2$ each have more
than $\ell$ vertices and $X^*$ has length less than $\ell$, we can choose $C_2$
so that $Y_1$ and $Y_2$ each have the largest possible length, and it will be
the length of the offshoots from $x$ containing them in $T$.  We determined
this maximum length $t'$ from $C'$ in Step 3.  We see in $C_2$ how many members
of $\{Y_1,Y_2\}$ have length $t'$.  Since we know $Q_1$, we know whether it has
length $t'$.  By knowing how many members of $\{Y_1,Y_2\}$ have length $t'$,
since they have the same length as the offshoots $Q_1$ and $Q_2^*$ in $T$ that
contain them, we know also whether $Q_2^*$ has length $t'$.

If $Q_2^*$ has length $t'$, then let $\cC_3$ be the family of subtrees of $T$
having a leg of length $q$ at the end of which is a single offshoot that
has length $t'$ and has at least $\C{Y_2}$ vertices.  In every member of
$\cC_3$, the vertex at distance $q$ from the leaf of the leg is $x$.  If also
$Q_1$ has length $t'$, then in some members of $\cC_3$ the offshoot arises by
deleting vertices from $Q_1$.  Since we know $Q_1$ and the number of peripheral
vertices outside $Q_1$, we can exclude any such
members, and a largest remaining member of $\cC_3$ shows us $Q_2^*$.

If $Q_2^*$ has length less than $t'$, then let $\cC'_3$ be the family of 
subtrees of $T$ having a leg of length $q$ at the end of which are two 
offshoots: one being a path of length $t'$ and one having $\C{V(Y_2)}$
vertices.  Because $\C{V(Y_2)}>\ell$, the branch vertex must be $x$.
Because we know $Q_1$, the number of paths of length $q+t'$, and the number
of peripheral vertices in $T_2$ (and $T_1$ if $j=(r+1)/2$), we can exclude the
members of $\cC'_3$ in which the offshoot with $\C{V(Y_2)}$ vertices comes from
$Q_1$.
A largest remaining member of $\cC'_3$ shows us $Q_2^*$.

Now that we know $Q_1$ and $Q_2^*$, we obtain the remaining offshoots at 
$x$ to reconstruct $T$.  If $\C{V(Q_2)}\ge \ell$, then let $\cC_4$ be the
family of subtrees of $T$ having a leg of length $q$ at the end of which
are two offshoots, one being $Q_1$.  The presence of $Q_1$ fixes the branch
vertex as $x$.  Since we know $Q_2^*$, and all the smaller offshoots at $x$
together have at most $\ell$ vertices, successively excluding the members of
$\cC_4$ in which the second offshoot comes from a known offshoot at $x$
allows us to find all the offshoots at $x$.

If $\C{V(Q_2)}<\ell$, then there may be multiple offshoots at $x$ that are
bigger then the $Q_2$ we see in $C'$.  In this case let $\cC'_4$ be the family
of subtrees of $T$ having a leg of length $q$ at the end of which are two
offshoots: one showing exactly $\ell+1$ vertices from $Q_1$ including a
path of length $q_1$, where $q_1$ is the length of $Q_1$.  This offshoot
can only be from $Q_1$, and the branch vertex is fixed at $x$.  Having
$\C{V(Q_2)}$ or fewer vertices in the second offshoot still fits in a card.
Since we know $Q_2^*$, we can therefore employ an exclusion argument using
$\cC'_4$ to find the remaining offshoots at $x$.
\end{proof}

We maintain the definitions and notation from Definition~\ref{longdef}.
By Corollary~\ref{s3spider}, we can reconstruct $T$ if some sparse card is a
$3$-legged spider.  Hence when $\Sp{(r-1)/2}$ is contained in $T$, we may
assume that it fits in a card, and from the deck we can recognize its presence.

\begin{lemma}\label{C1fixz}
Let $T$ be a tree with $n\ge 6\ell+\c$ and $r\le n-3\ell-1$ that has a sparse
card, with optimal sparse cards having primary vertex $z$ and primary value $j$,
and let $t$ be the maximum length of offshoots from $P$ in an optimal sparse
card.

(a) $T$ has a long subtree $C_0$ with $r$-path $P'$ having vertices $\VEC u1r$
such that offshoots from $P'$ at $u_j$ in $C_0$ together have at least $\ell+1$
vertices and offshoots from $P'$ in $C_0$ at other vertices have at most $\ell$
vertices.  Furthermore, any such subtree $C_0$ satisfies $u_j=z$.

(b) If a subtree $C_0'$ with the properties of $C_0$ has exactly 
$\ell+1$ vertices in offshoots from $P'$ at $u_j$, then $C_0'$ fits in a card.
Furthermore, a largest such $C_0'$ shows in full the component of $T-z$
containing $u_{j+1}$.  If $t<j-1$, then it also shows the component containing
$u_{j-1}$.

(c) If $T$ has a sparse card with degree $3$, then $T$ has a subtree $C_1$ as
in (a) such that also the first branch vertex on $P'$ is $u_j$ and one of the
offshoots from $P'$ at $u_j$ has exactly $\ell+1$ vertices (others have at most
$\ell$ vertices.) Furthermore, if $j<(r+1)/2$ or $\Sp{(r-1)/2}\nosub T$, then
the component of $T-z$ containing $u_{j+1}$ is seen in full in a largest such
$C_1$.
\end{lemma}

\begin{proof}
Let $W$ be the union of the offshoots from $P$ in an optimal sparse card $C$.
If $T$ has a sparse card of degree $3$, then we may assume that $C$ has degree
$3$ and $W$ is a single offshoot, by Lemma~\ref{sparse3}.  As in
Definition~\ref{longdef}, we have $z=v_j$ with $P=\la\VEC v1r\ra$.

Note first that subtrees with the properties required for $C_0$ and $C_1$
exist.  Simply include $P$ and exactly $\ell+1$ vertices from $W$ (offshoots
from $P$ at $\VEC v{j+1}r$ may be included but are not needed).  Here $z$
serves as $u_j$ and is the first branch vertex on $P$, which serves as $P'$.

\smallskip
{\it Proof of (a).}
Consider such a subtree $C_0$, and let $W'$ be the subgraph of $C_0$ induced
by the vertices in offshoot(s) from $P'$ at $u_j$.
If $u_j\notin V(P)$, then to avoid having a path longer than $P$, the
path from $u_j$ to $V(P)$ must start along $u_ju_{j+1}$ and reach $P$ at a
vertex other than $z$.  This puts $W'$ into an offshoot from $P$ not at $z$.
Such offshoots have at most $\ell$ vertices, but $W'$ has more than $\ell$
vertices; we conclude $u_j\in V(P)$.  If $u_j\in V(P)-\{z,v_{r+1-j}\}$, then
we can choose paths from $P$ and $P'$ ending at $u_j$ whose union is a path
longer than $P$; that is, when $u_j\in\{\VEC v{j+1}{r-j}\}$ we replace
$\la \VEC u1j\ra$ in $P'$ with part of $P$ that is longer, and when
$u_j\notin\{\VEC vj{r+1-j}\}$ we replace the shorter portion of $P$ with part
of $P'$ that is longer.  Hence $u_j\in\{z,v_{r+1-j}\}$.  If $u_j=v_{r+1-j}$
with $j\ne(r+1)/2$, then $W'$ plus at least $r-j$ vertices of $P'$ lie in
$T_2$, contradicting that $T_2$ has at most $\ell+r-j$ vertices.  Hence $u_j=z$.

\smallskip
{\it Proof of (b).}
By part (a), $u_j=z$ in $C_0'$.  Since $t<j-1$ and $u_j=z$, the path $P'$
cannot enter $W$.  Since $T$ has at most $\ell$ vertices outside $C\cup W^*$,
the offshoots $W'$ from $P'$ at $u_j$ in $C_0'$ must lie in $W^*$.  Since
$r\le n-3\ell-1$, the offshoot $W^*$ has at least $2\ell+1$ vertices, so
$C_0'$ omits at least $\ell$ vertices from $W^*$ and fits in a card.  Outside
$W'$, a largest such $C_0'$ shows us everything in $T$ except $W^*$ and other
offshoots at $z$.

\smallskip
{\it Proof of (c).}
Recall that we use $T_2$ to denote the component of $T-z$ containing $v_{j+1}$.
Consider $C_1$, in which $u_j=z$ by part (a).  Let $W'$ denote the offshoot
from $P'$ at $z$ in $C_1$ that has $\ell+1$ vertices.  In the case $j<(r+1)/2$,
any path of length $r-j$ extending from $z$ lies in $T_2$, so
$\VEC u{j+1}r\in V(T_2)$.  Because $u_j$ is the first branch vertex on $P'$,
the subset of $V(C_1)$ consisting of $\VEC u1{j-1}$ and $V(W')$ contains
exactly $j+\ell$ vertices from $P'\cup W'$.  Since no path from $z$ outside
$T_2$ has length more than $j-1$, and the only offshoot from $P$ at $z$ in
$T$ with more than $\ell$ vertices is $W^*$, we conclude that $C_1$ contains at
most $j+\ell$ vertices from $T_1\cup W^*$.  However,
$\C{V(W^*)}\ge \C{V(W)}\ge 2\ell+1$, so $T_1\cup W^*$ has at least $j+2\ell$
vertices.  Hence $C_1$ omits at least $\ell$ vertices and fits in a card, we
see it as a csc, and a largest candidate for $C_1$ shows us all of $T_2$.

Similarly, if $j=(r+1)/2$ and $\Sp{(r-1)/2}\nosub T$, then  $W^*$ does not have
a path of length $(r-1)/2$ from $z$, and so $P'$ must be contained in
$T_1\cup T_2$.  Now $W^*$ has at least $2\ell+1$ vertices and at most $\ell+1$
of them appear in $C_1$, so again a largest such csc shows us all of $T_2$,
where we have indexed $T_1$ and $T_2$ so that $\C{V(T_2)}\ge\C{V(T_1)}$.
If they have the same size, then there are two such largest cscs, and we see
one of $T_1$ and $T_2$ in each of them.
\end{proof}

When $r$ is odd, $T$ has a unique central vertex.  In this case, we call the
subtrees obtained by deleting the central vertex the {\it pieces} of $T$.
We previously used this term in~\cite{KNWZ3}.

\begin{lemma}\label{nolong}
Suppose $n\ge6\ell+\c$ and $r\le n-3\ell-1$.  If $T$ has a sparse card of degree
$3$, and $T$ has no long-legged card, then $T$ is $\ell$-reconstructible.
\end{lemma}
\begin{proof}
With notation as in Definition~\ref{longdef}, by Lemma~\ref{sparse3} we
can choose an optimal sparse card $C$ of degree $3$ to maximize the length
of $W$, which we call $t$.  Note that $t\le j-1$, since $P$ is a longest path
in $T$.  By Corollary~\ref{s3spider}, we may assume that $C$ is not a spider.
Hence $W$ is not a path.  This implies that $t$ is also the length of the
offshoot $W^*$ in $T$, since otherwise vertices from a longer path in $W^*$
could be added while deleting leaves of $W$ not on that path.  Since
$r\le n-3\ell-1$, there are at least $2\ell+1$ vertices in $W$.  We consider
several cases in terms of $t$, $j$, and $\ell$.  We know these values, so we
can recognize the cases.

\medskip
{\bf Case 1:} {\it $t<j-1$.}
Since the offshoot $W$ in $C$ has $n-\ell-r$ vertices, we can find $C_0'$ as in
Lemma~\ref{C1fixz}(b) in which one offshoot from $u_j$ has exactly $\ell+1$
vertices.  By Lemma~\ref{C1fixz}(b), a largest such $C_0'$ shows us $T_2$ and
$T_1$.  In fact, since the $\ell+1$ vertices in the offshoot from $u_j$ omit
at least $\ell$ vertices from $W$, we can also allow $C_0'$ to have other
offshoots at $z$ and make the same argument.  Now a largest such $C_0'$ shows
us everything in $T$ except $W^*$.

Since we now know all of $T$ except $W^*$, we know $\C{V(W^*)}$ and also the
common length of $W$ and $W^*$, which we are calling $t$.  If $T$ has a
long-legged card, then we see it and can apply Lemma~\ref{longleg} to
reconstruct $T$.  Hence we may assume that $T$ has no long-legged card.
Now $\C{V(W^*)}-\C{V(W)}<j-1$; otherwise, replacing
$\VEC u1{j-1}$ with vertices of $W^*-V(W)$ would produce a long-legged card.

Let $i=1+\C{V(W^*)}-\C{V(W)}$; thus $i\le j-1$.  Let $\cC_2$ be the family of
cards consisting of a path $P''$ with vertices $\VEC uir$ and an offshoot
$W''$ from $P''$ at $u_j$ that has $\C{V(W^*)}$ vertices and length $t$.  It
suffices to find a member of $\cC_2$ in which $u_j=z$ and $W''$ is $W^*$.

We prove first that $u_j=z$ for any card in $\cC_2$ (the argument is like
that in Lemma~\ref{C1fixz}(a)).  If $u_j\notin V(P)$, then the path from $u_j$
to $V(P)$ starts with $u_ju_{j+1}$ and reaches $P$ in $T_2$, since $T$
has no path longer than $P$.  That puts $W''$ into an offshoot from $T_2$,
contradicting that there are at most $\ell$ vertices in offshoots from $T_2$.
Hence $u_j\in V(P)$.  If $u_j\in V(P)-\{z\}$, then since $W''$ is too big to be
in an offshoot from $P$, the path $P$ must enter $W''$.  Since $t<j-1$ and
$T_2$ has at most $\ell$ vertices in offshoots from $P$, the part of $P$ in
$W''$ must come from $T_1$.  Since any vertex of $W''$ starts a path of length
more than $r-j$, we have $z\notin W''$.  Hence $z$ has degree $2$ in the card
and $W''\esub T_1$.  Since $T_1$ has at most $j-1+\ell$ vertices and the card
has only $\la\VEC u{j+1}r\ra$ from $T_2$, the card has fewer than $r+\ell$
vertices, a contradiction.

Now consider the relationship between $j-i$ and $t$.  If $j-i>t$, then the path
$\la \VEC uij\ra$ in a card in $\cC_2$ cannot lie in $W^*$, since $u_j=z$ and
$W^*$ has length only $t$.  Hence this path must come from $T_1$, and any
card in $\cC_2$ shows us $W^*$ to complete the reconstruction of $T$.

Hence we may assume $j-i\le t$.  We have proved $u_j=z$ for every card in
$\cC_2$.  Some cards in $\cC_2$ may take $\la \VEC ui{j-1}\ra$ from $W^*$ and
$W''$ from $T_1$.  After excluding these, the remaining cards in $\cC_2$ take
$\la \VEC ui{j-1}\ra$ from $T_1$ and show us $W^*$ as $W''$.

Knowing $T_1$, we know all subtrees of $T_1$ having $\C{W^*}$ vertices, with
multiplicity.  Knowing $T_2$, we know the number of peripheral vertices in
$T_2$.  Hence we will know all the cards in $\cC_2$ to be excluded if we can
determine the number of vertices in $W^*$ at distance $j-i$ from $z$.

Since $t<j-1$ and $T_1$ has length $j-1$, a copy of the spider
$S_{j-i,j-1,r-j}$ whose branch vertex is $z$ must have its leaves in
$W^*$, $T_1$, and $T_2$, respectively (where we are free to make that choice
for each copy of $S_{j-i,j-1,r-j}$ when $j=(r+1)/2$).  Hence by counting
these spiders we can compute the number of vertices in $W^*$ at distance
$j-i$ from $z$ if we can exclude from the count of these spiders all instances
where the branch vertex is not $z$.

Let $S$ be a copy of $S_{j-i,j-1,r-j}$ in $T$ with $r$-path $P'$ and branch
vertex $z'$ other than $z$.  Having $z'$ in $T_1$ or $W^*$ would create a path
longer than $P$, as would having $z'$ in $T_2$ at a distance from $z$ other
than $r-2j+1$.  This implies $z'=z$ if $j=(r+1)/2$, so we may assume
$j<(r+1)/2$.  Now the leaf of the long leg of $S$ must be a peripheral
vertex in $T_1$, since $t<j-1$.  Since we know both $T_1$ and $T_2$, and
$S$ is contained in their union with $z$, we can count all such $S$.
As noted earlier, from the remaining copies of $S_{j-i,j-1,r-j}$ we can count
the vertices in $W^*$ at distance $j-i$ from $z$, which allows us to exclude
the bad members of $\cC_2$ and complete the reconstruction.

\medskip
{\bf Case 2:} {\it $t=j-1=(r-1)/2$.}
This case requires $\Sp{(r-1)/2}\esub T$.
All $r$-paths are longest paths in $T$ and have the same center, which is
the primary vertex $z$ in $C$.  Among largest long cscs in which the center $z$
has degree $2$, choose $\hC$ to be one having a largest piece (a component of
$\hC-z$).  Since we are given that $T$ has no long-legged card, this largest
piece of $\hC$ is a piece of $T$ seen in full (if it does not appear in full
and $\hC$ is not a card, then we can grow that piece; if $\hC$ is a card, then
the other piece in $\hC$ has a leaf available to shift to it).  If the two
pieces of $\hC$ have the same size, then both are seen in full.

Since $W$ has $n-r-\ell$ vertices, and we can make a candidate for $\hC$ by
appending $z$ and a path of length $(r-1)/2$ to $W$, the large piece in $\hC$
has at least $n-r-\ell$ vertices.  Now we can make an optimal sparse card by
reducing the large piece in $\hC$ to $n-r-\ell$ vertices and attaching what
remains to the center of an $r$-path, since $T_1$, $T_2$, and $W^*$ all contain
paths of length $(r-1)/2$ from $z$.  Therefore, we may assume that $C$
was chosen so that the largest piece in $\hC$ is in fact $W^*$.

\smallskip
{\bf Subcase 2a:} {\it $\hC$ has fewer than $n-\ell$ vertices.}
Here both pieces of $\hC$ are largest pieces of length $(r-1)/2$ in $T$, seen
in full; we may assume that $C$ was chosen so that that $W^*$ is one of them
and $T_2$ is the other.  Hence we now know $W^*$ and $T_2$.  We also know how
many vertices lie outside $W^*\cup P$; it is at most $\ell$.  These vertices
are distributed among $T_1$, possibly other pieces of length $(r-1)/2$, and
possibly shorter offshots at $z$.

For each isomophism class of pieces of length $(r-1)/2$ having size
$\C{V(W^*)}$, if there are $k$ such pieces in $T$ then there are $\CH k2$
choices for $\hC$ in which this isomorphism class provides both pieces.
Hence we know the multiplicities of all the largest pieces of length $(r-1)/2$.
Let $\cC$ be the family of all long cscs with degree $2$ at $z$ in which
one of the two pieces is $W^*$ and the other is smaller.  Because we know the
largest pieces with multiplicity, we can exclude all members of $\cC$ in which
the smaller piece is obtained by deleting vertices from one of the largest
pieces.  A largest remaining member $\cC$ shows us a smaller piece of length
$(r-1)/2$, if it exists.  Continuing the exclusion argument allows us to find
all pieces of length $(r-1)/2$.

Once we know all the pieces of length $(r-1)/2$, we know the numbers of
peripheral vertices in each, and we know the number of vertices remaining
in the shorter pieces; let this be $s$.  Since the vertices in the shorter
pieces are absent from the card $C$, we have $s\le \ell$.  Let $\cC'$ be
the family of subtrees consisting of an $r$-path plus at most $s$ vertices in a
single offshoot from the center; every such subtree is smaller than $C$ and
fits in a card.  After excluding all members of $\cC'$ in which the offshoot
arises by deleting vertices from a piece of length $(r-1)/2$, a largest
remaining member shows us one of the shorter offshoots at $z$.  Continuing
the exclusion argument on $\cC'$ shows us all the remaining offshoots.

\smallskip
{\bf Subcase 2b:} {\it $\hC$ is a card.}  Since $\Sp{(r-1)/2}\esub T$, with
center at $z$, the card $\hC$ omits at least $(r-1)/2$ vertices.  Hence
$\ell\ge(r-1)/2$, so $r\le 2\ell+1$.  Recall that $\C{V(W)}\ge 2\ell+1$.

With $(r-1)/2\le \ell-1$ and $\C{V(W)}\ge 2\ell+1$, in $W$ there are at least
$\ell+1$ vertices outside a given path of length $(r-1)/2$ from $z$.  Let $C_2$
be a long subtree that is largest among those that contain $\Sp{(r-1)/2}$
and contain offshoots with exactly $\ell+1$ vertices from one of the legs
of $\Sp{(r-1)/2}$ (the other legs and center may also have offshoots).  Since
$C$ is a card, we know that every component of $T-z$ other than $W^*$ has at
most $(r-1)/2+\ell$ vertices, while $W^*$ has more than $(r-1)/2+\ell$.  Hence
in $C_2$ the leg of the spider having $\ell+1$ vertices in offshoots must come
from $W^*$.

The subtree $C_2$ contains everything in $T$ that is outside $W^*$.
That includes $P$ and everything outside $C\cup W^*$.  Since $C$ is a card,
there are at most $\ell$ vertices outside $C\cup W^*$.  Including $P$ with
$r$ vertices, outside $W^*$ there are at most $3\ell+1$ vertices in $C_2$.
Also $C_2$ has exactly $(r-1)/2+\ell+1$ vertices from $W^*$, and this
contribution is at most $2\ell+1$.  Hence $C_2$ has at most $5\ell+2$
vertices, fits in a card (since $n-\ell\ge 5\ell+10$), and shows us all of $T$
outside $W^*$.  We know $W^*$ earlier from $\hC$, so we have reconstructed $T$.

\medskip
{\bf Case 3:} {\it $t=j-1<(r-1)/2$ and $t\ge \ell+1$.} From the card $C$ we
know that offshoots from $z$ other than $T_1$, $T_2$, and $W^*$ together have
at most $\ell$ vertices.  Consider the subtree $C_1$ obtained in
Lemma~\ref{C1fixz}(c); we have $u_j=z$, and $C_1$ fits in a card, and $C_1$
shows all of $T_2$.  We see all of $T$ in full in $C_1$ except $W^*$ and a
smallest offshoot from $z$ of length $j-1$, which we may assume is $T_1$.  In
$C_1$ we have $W'$ and the path $\la\VEC u1{j-1}\ra$ as offshoots from $u_j$;
one is contained in $W^*$ and the other in $T_1$, but we do not know which is
where.

Let $\cC_3$ be the family of long subtrees of $T$ whose vertices in positions
$j$ and $r+1-j$ on an $r$-path $P''$ have degree $2$.  Let $C_3$ be a largest
member of $\cC_3$.  Any member of $\cC_3$ has $z$ in position $j$ or $r+1-j$
along $P''$, since $C$ prevents offshoots from $P$ having more than $\ell$
vertices (other than $W^*$), and here $j-1\ge\ell+1$ forces such an offshoot if
$z$ is elsewhere.  Since $z$ has degree $2$ in $C_3$ and starts a path of
length $r-j$ in $C_3$, the subtree $C_3$ must be missing $T_1$ or $W^*$.
Since $j-1>\ell$, this means that $C_3$ (and any member of $\cC_3$) fits in a
card.

Since we know $T_2$, we know the subtrees of $T_2$ that can arise when a vertex
at distance $r+1-2j$ from $z$ is restricted to degree $2$.  The part of $C_3$
outside such a subtree is a copy of $T_1$ or $W^*$, whichever is larger.  If
they have the same size, then we get both from two candidates for $C_3$.
Otherwise, we use an exclusion argument to find the other.  Having obtained the
larger one, say $X$, we know the number $n'$ of vertices in the smaller one,
say $Y$.  Knowing $X$, we know all the rooted subtrees of $X$ with $n'$
vertices.  Excluding from $\cC_3$ the members that arise using such a subtree
from $X$ on one end and a largest contribution from $T_2$ in the middle and
other end, a member of $\cC_3$ of the right size that remains shows us $Y$.

\medskip
{\bf Case 4:} {\it $t=j-1<(r-1)/2$ and $t\le\ell$.}
Let $C'$ be a largest long subtree of $T$ among those having a leg with length
more than $r-j$.  Since by assumption $T$ has no long-legged card, $C'$ fits in
a card.  If $C'$ does not enter $W^*$, then it has at most $r+\ell$ vertices,
since only $\ell$ vertices are outside $C$.  Using $W^*$ yields a choice for
$C'$ with at least $r-j+1+2\ell+1$ vertices; such $C'$ is bigger since
$j-1\le\ell$.  Hence the csc $C'$ shows us $W^*$ in full.

In a largest subtree $C_1$ as obtained in Lemma~\ref{C1fixz}(c), we have $u_j=z$
and see all of $T_2$.  As in Case 3, in $C_1$ we see all offshoots from $z$
other than a smallest offshoot of length $j-1$, which we may take to be $T_1$.
Again the offshoot $W'$ in $C_1$ may lie in $W^*$ or in $T_1$.

Having obtained $W^*$ from $C'$, we now obtain $T_1$ by an exclusion argument.
Since $C$ omits exactly $\ell$ vertices and $j-1\le \ell$, at most $2\ell$
vertices are in $T_1$, while $W^*$ has at least $2\ell+1$.  Knowing $T_2$ and
$W^*$
(and extra offshoots at $z$), we know $\C{V(T_1)}$; call it $n'$.  We
also know the numbers of vertices in $W^*$ and $T_2$ that are farthest from
$z$ and the number of $r$-vertex paths in $T$, so we can compute the number of
vertices in $T_1$ that are farthest from $z$.

Now consider cscs with $\C{V(C')}-(\C{V(W^*)}-n')$ vertices having an
$r$-vertex path and a leg with length more than $r-j$.  Exclude all those in
which the portion outside the leg comes from the known subtrees $W^*$ or $T_2$.
Such a csc that remains shows us $T_1$, completing the reconstruction of $T$.
\end{proof}

\begin{lemma}\label{nospar3}
If $n\ge6\ell+\c$ and $r\le n-3\ell-1$, and $T$ has a sparse card but
is not $\ell$-reconstructible, then $T$ has no sparse card of degree $3$,
every sparse card is optimal, and $T$ has no long-legged card.
\end{lemma}
\begin{proof}
All claims except the last one follow from Lemmas~\ref{optcard}, \ref{longleg},
and~\ref{nolong}.  Hence we assume these claims and prove that $T$ has no
long-legged card.  Let $C$ be a sparse card, and define notation as in
Definition~\ref{longdef}.

Since $r\le n-3\ell-1$, a long-legged card $C'$ has at least $2\ell+1$
vertices outside a longest path in $T$, such as the $r$-path $P$ in $C$.
However, $T_1\cup T_2$ has at most $\ell$ vertices outside $P$.  Hence $C'$
must have at least $\ell+1$ vertices in offshoots from $P$ at $z$.  Since the
length of a leg $L$ in $C'$ exceeds $r-j$, either $z$ lies on $L$ or is not in
$C'$, so all such vertices in $C$ lie in only one offshoot from $P$ at $z$.

Let $z'$ be the vertex of $C'$ closest to $z$; since no path of length more
than $r-j$ starts at $z$, vertex $z'$ must lie on the leg in $C'$.
Now deleting the portion of $L$ from $z'$ to the leaf and adding the path
from $z'$ to $z$ and all of $P$ yields a tree with more vertices than $C'$
in which $z$ has degree $3$.  Iteratively deleting leaves from $C'$ in this
trees yields a sparse card in which the primary vertex $z$ has degree $3$,
which we have observed does not exist.
\end{proof}

A {\it paddle} is a long card having a leg
of length at least $\CL{r/2}$.  In particular, when $r$ is odd, a paddle
is a long card in which the center has degree $2$ and one piece is a path.
The next lemma will be used in each of the two final lemmas.

\begin{lemma}\label{nopaddle}
If $r$ is odd, and $T$ contains no paddle and no sparse card of degree $3$,
then $T$ is $\ell$-reconstructible.
\end{lemma}
\begin{proof}
Let $y$ be the unique center of $T$.
All long subtrees have $y$ as the unique center.  Among largest long
cscs with $y$ having degree $2$, choose $\hC$ to maximize the number of 
vertices in the larger piece.  Let $W_1$ and $W_2$ be the pieces,
with $\C{V(W_1)}\ge\C{V(W_2)}$.  By the hypothesis, $W_1$ is a component
of $T-y$, seen in full in $\hC$ (if it does not appear in full and
$\hC$ is not a card, then we can grow that piece; if $\hC$ is a card, then
since $T$ has no paddle the piece other than $W_1$ in $\hC$ has at least two
leaves, and one can be deleted to allow $W_1$ to grow).

If $\hC$ is not a card, then $W_1$ and $W_2$ are two largest offshoots of
length $(r-1)/2$ from $y$ in $T$.  If there are $k$ copies of a largest
offshoot of length $(r-1)/2$, then we know this because there are $\CH k2$
choices for the card $\hC$ having it as both offshoots.  Thus we find all the 
largest offshoots of length $(r-1)/2$, with multiplicity.  By considering cscs
that fix the largest offshoot of length $(r-1)/2$ as one of these, we can then
find the other offshoots of length $(r-1)/2$ from $y$ by an exclusion argument.

To find the shorter offshoots, consider subtrees consisting of an $r$-path
plus one offshoot from the central vertex $y$.  Since $T$ has no sparse card
of degree $3$, every such subtree fits in a card.  Hence we know the largest
such offshoots, whether they have length $(r-1)/2$ or not.  We can therefore
find all the offshoots by an exclusion argument.

Hence we may assume that $\hC$ is a card.
Among long cards with degree $2$ at $y$, choose $\hC'$ to maximize the number
of vertices in the smaller of the two pieces.  Let $W'_1$ and $W'_2$ be the
two pieces, with $\C{V(W'_1)}\ge\C{V(W'_2)}$.

If $\C{V(W'_2)}\le\C{V(W'_1)}-2$, then $W'_2$ is the second largest offshoot
of length $(r-1)/2$ from $y$ in $T$, seen in full.  To find the other
components of $T-y$, consider the family $\hat\cC$ of long cscs having center
of degree $2$, where both pieces have length $(r-1)/2$, and one of the pieces
is $W'_1$.  Since we know $W'_2$, an exclusion argument finds all the offshoots
from $y$ in $T$ having length $(r-1)/2$.  For the shorter offshoots, we
consider the family $\hat\cC'$ of long cscs consisting of one $r$-path and one
offshoot from the path at its center.  Since $T$ has no sparse card of degree
$3$, we see all such offshoots in cscs; starting with the largest (and those of
length $(r-1)/2$), an exclusion argument applies.

Hence we may assume $\C{V(W'_2)}\ge\C{V(W'_1)}-1$, so
$\C{V(W'_2)}\ge(n-\ell-2)/2\ge\FR32(\ell+1)$, where the last inequality
only needs $n\ge4\ell+5$.  Since $\C{V(W'_2)}\ge\ell+1$, we may let $\hC_1$ be
a largest long subtree having two offshoots of length $(r-1)/2$ from the center
$y$ that each have $\ell+1$ vertices (``largest'' indicates that the number of
vertices in the other offshoots is maximized).  Note that $W'_2$ may not be a
component of $T-y$ in full, but we do know $W_1$ from $\hC$.  Let $W^*_2$ be a
largest offshoot of length $(r-1)/2$ from $y$ other than (one copy of) $W_1$.
Since $\hC$ is a card, only two components of $T-y$ contain offshoots from $y$
with at least $\ell+1$ vertices, so one of the pieces of $\hC_1$ with $\ell+1$
vertices lies in $W_1$ and the other in $W^*_2$.

Since each of these
two pieces in $\hC_1$ omits at least $(\ell+1)/2$ vertices of the component of
$T-y$ containing it, $\hC_1$ fits in a card and shows us all of $T-W_1-W^*_2$.
Since we also know $W_1$, this tells us the number of vertices in $W^*_2$ and
the number of peripheral vertices in each offshoot from $y$ other than $W^*_2$.
Since the $(n-\ell)$-deck determines the $k$-deck when $k<n-\ell$, we also know
the total number of $r$-paths in $T$.  From these quantities, we can compute
the number of peripheral vertices in $W^*_2$.

Now let $\cC^*$ be the family of all long cscs with central vertex of degree
$2$ such that one offshoot from the center is a path and the other has
$\C{V(W^*_2)}$ vertices.  Since $T$ has no paddle, the requested number of
vertices is less than $n-\ell$, so $\cC^*$ is nonempty.  In each member of
$\cC^*$, the offshoot with $\C{V(W^*_2)}$ vertices either is contained in
$W_1$ or is $W^*_2$.  Since we know $W_1$ and know the number of peripheral
vertices in each offshoot at $y$, we know how many times each subgraph of $W_1$
with $\C{V(W^*_2)}$ vertices appears in a member of $\cC^*$.  By excluding
them, we find $W^*$ and thus have reconstructed $T$.
\end{proof}

\begin{lemma}\label{sparseno3}
Suppose $n\ge6\ell+\c$ and $r\le n-3\ell-1$.  If $T$ has a sparse card,
then $T$ is $\ell$-reconstructible.
\end{lemma}
\begin{proof}
As noted in Lemma~\ref{nospar3}, we may assume that $T$ has no sparse
card of degree 3 and every sparse card is optimal, with the same primary vertex
$z$ and primary value $j$.  Since all sparse cards are visible in the deck, we
recognize such $T$.

Among all sparse cards, choose $C$ and its $r$-path $P$ to maximize the length
$t$ of a longest offshoot from $P$ at the primary vertex, and within that
maximize the number of vertices in a largest offshoot of length $t$.  We use
notation relative to $C$ as in Definition~\ref{longdef}.  Note that $t\le j-1$.
Since we see all sparse cards, we know $t$.

Suppose first that $t=(r-1)/2$, which requires $j=(r+1)/2$.
By Lemma~\ref{nospar3}, $T$ has no long-legged card, which in the case
$j=(r+1)/2$ is the same as saying that $T$ has no paddle.
Hence Lemma~\ref{nopaddle} applies to show that $T$ is $\ell$-reconstructible.

Hence we may assume $t<(r-1)/2$.
Consider a largest subtree obtained as $C_0'$ in Lemma~\ref{C1fixz}(b), so
$C_0'$ fits in a card, with $u_j=z$, and $C_0'$ shows us the subtree $T_2$ in
full as the component of $C_1-z$ containing $u_{j+1}$ (we see both $T_1$ and
$T_2$ in two such subtrees if $j=(r+1)/2$ and they have the same size).

\medskip
{\bf Case 1:} {\it $t<j-1$}.
In this case $C_0'$ also shows us $T_1$, by Lemma~\ref{C1fixz}(b).
We now know all of $T$ except some of the offshoots from $P$ at $z$.  We know
the total number $s$ of vertices in offshoots from $P$ at $z$, and we know
their maximum length, $t$.

Let $\cC_2$ be the family of subtrees consisting of an $r$-path 
$\la \VEC u1r\ra$ and one offshoot at $u_j$.  Since $T$ has no sparse card of
degree $3$, all members of $\cC_2$ have fewer than $n-\ell$ vertices.  Since we
know $T_1$ and $T_2$, we know the number of peripheral vertices in those
subtrees, and there are no other peripheral vertices in $T$ (since $t<j-1$).
Since we know $T_2$, we also know all the members of $\cC_2$ in which the
offshoot at $u_j$ arises from $T_2$.  In the remaining members, $u_j=z$ and the
$r$-path comes from $T_1\cup T_2$.  Hence largest remaining members of $\cC_2$
show us offshoots from $P$ at $z$ in $T$, seen in full.  Excluding subtrees
from such offshoots with each smaller size, the exclusion argument now yields
all the offshoots from $P$ at $z$ in $T$, which completes the reconstruction of
$T$.

\medskip
{\bf Case 2:} {\it $t=j-1$ with $j<(r+1)/2$.}
We have noted that $C_0'$ satisfies $u_j=z$ and shows us $T_2$.
The path $\la \VEC u1{j-1}\ra$ in $C_0'$ may come from $T_1$ or from $W$.

First we will use an exclusion argument to find all offshoots from $z$ in $T$
with length $j-1$.  Let $\cC_3$ be the family of long subtrees that have
exactly one offshoot of length $j-1$ from the vertex $u_j$ in position $j$ on
a specified $r$-path (and no other offshoot at $u_j$).  Since $T$ has no sparse
card of degree $3$, all members of $\cC_3$ fit in cards, and we see them.
Since we know $T_2$ and the number of $r$-paths in $T$, and since by
$j<(r+1)/2$ every $r$-path has exactly one endpoint in $T_2$, we know
the number of peripheral vertices not in $T_2$.  We can exclude (with
multiplicity) the members of $\cC_3$ in which $u_j$ and the offshoot come from
$T_2$ (in particular, when $u_j=v_{r+1-j}$).  In the remaining members of
$\cC_3$, we have $u_j=z$.  Largest remaining members show us largest offshoots
from $z$ having length $j-1$.  Let $W'$ be such an offshoot.

Knowing $T_2$, the largest such offshoots (with multiplicity), and the number
of peripheral vertices, we know the number of peripheral vertices outside $W'$.
Therefore, we know the members of $\cC_3$ in which the offshoot from the
$r$-path comes from $W'$.  After excluding all members generated from
largest offshoots such as $W'$ in this way, a largest remaining member of
$\cC_3$ shows us another offshoot of length $j-1$.  Continuing the exclusion
argument shows us all offshoots of length $j-1$ from $z$ in $T$, including
$T_1$.

Finally, let $\cC_4$ be the family of long subtrees having exactly one offshoot
from the vertex in position $j$ on the specified $r$-path, without regard to
the length of the offshoot.  Since $T$ has no sparse card of degree $3$,
each member of $\cC_4$ fits in a card, and we see it.  Since we know all the
peripheral vertices in offshoots at $z$, we can exclude the members of $\cC_4$
in which the offshoot is contained in $T_2$ or in any offshoot at $z$ having
length $j-1$.  A largest remaining member shows us a largest offshoot at $z$
among those with length less than $j-1$.  Continuing the exclusion argument
shows us the remaining offshoots at $z$.
\end{proof}

\begin{lemma}\label{nosparse}
Suppose $n\ge6\ell+\c$ and $r\le n-3\ell-1$.
If $T$ has no sparse card, then $T$ is $\ell$-reconstructible.
\end{lemma}

\begin{proof}
From the deck we recognize that $T$ has no sparse card.
Recall that a paddle is a long card having a leg of length at least $r/2$.

Suppose first that $T$ has no paddle.  If $r$ is odd, then Lemma~\ref{nopaddle}
applies and $T$ is $\ell$-reconstructible.  If $r$ is even, then let $C'$ be a
largest long subtree among those having a leg of length at least $r/2$.  Since
$T$ has no paddle, $C'$ fits in a card and shows in full one branch of $T$
obtained by deleting the central edge; call it $T'$.  If there are distinct
choices for $C'$, then the two branches have the same size and we see them.
Otherwise, knowing $T'$ and the total number of $r$-paths gives the number of
peripheral vertices outside $T'$, which with the number of copies of $C'$ tells
us whether the two branches are isomorphic.  If the branches are not
isomorphic, then they do not have the same size, and we know the number $n'$ of
vertices in the smaller branch.  Among the long cscs with one branch being a
path and the other having $n'$ vertices, we can exclude those whose non-path
branch comes from $T'$.  A largest remaining such csc shows the other branch,
completing the reconstruction.

Hence we may assume that $T$ has a paddle.  Among paddles whose long leg has
maximum length, choose $C^*$ to maximize the length of a longest offshoot
from the $r$-path $P$ at the branch vertex $z$ at the end of the leg.  Within
this, choose $C^*$ to maximize the total number of
vertices in the offshoots from $P$ at $z$.  With $P=\la\VEC v1r\ra$, define $j$
by $z=v_j$ with $j<(r+1)/2$ (the long leg implies $j\ne(r+1)/2$).
Although we have no sparse card, we still use ``$z$'' and ``$j$'' here, because
the items so designated in this argument play roles analogous to those played
by the items with those names earlier in this section.

Let $Q$ be the union of the offshoots in $T$ from $P$ at $z$.  Let $T_1$ and
$T_2$ be the components of $T-z$ containing $v_{j-1}$ and $v_{j+1}$,
respectively.  Since $T_2$ contributes only a path to $C^*$ and $T$ has no
sparse card, $T_1$ cannot contribute only a path $C^*$.  Hence $C^*$ shows us
$Q$ in full, since otherwise we could delete a vertex of $T_1$ from $C^*$ and
enlarge another offshoot.  Note that $T_1$ is a smallest offshoot from 
$T_2$ at $z$ among those with maximum length.

If we can determine $T_2$, then finding $T_1$ will complete the reconstruction.
To see that, let $\cC_1$ be the family of long subtrees having a leg of length
at least $r-j+1$.  By the choice of $C^*$ as a card with longest leg, every
member of $\cC_1$ has fewer than $n-\ell$ vertices, so we know all these
subtrees.  Knowing $T_2$ and $Q$, we can compute the number of peripheral
vertices in $T_1$ (whether or not $Q$ has length $j-1$).  We can then exclude
from $\cC_4$ all members where the portion obtained by deleting the leg comes
from $T_2$ or $Q$.  A largest remaining member of $\cC_4$ shows us $T_1$.
Hence it suffices to find $T_2$.

\smallskip
{\bf Case 1:} {\it $Q$ has in total at least $\ell+1$ vertices.}
Let $C_2$ be a largest subtree containing an $r$-vertex path $P'$ with vertices
$\VEC u1r$ such that $\la\VEC u1j\ra$ is a leg of $C_2$ and the offshoots from
$P'$ at $u_j$ in $C_2$ form a specified rooted subgraph of $Q$ with exactly
$\ell+1$ vertices.  There is such a subtree with $u_j=z$.  Since $C^*$ is a
card, there are at most $\ell$ vertices in offshoots from any path of length
$r-j$ in $T_2$.  Hence the $\ell+1$ vertices from $Q$ guarantee that $u_j$ does
not lie in $T_2$.  Also $u_j$ cannot lie in $T_1\cup Q$, since $C_2$ has two
paths of length at least $j-1$ from $u_j$.  Hence $u_j$ must equal $z$.

Since $u_j=z$ and there is only one offshoot from $z$ with length $r-j$,
the path $\la \VEC u{j+1}r\ra$ in $C_2$ must lie in $T_2$.
In $C_2$, there are exactly $j+\ell$ vertices from $T_1\cup Q$.
From $C^*$, we know that $T_2$ has at most $r-j+\ell$ vertices.
Hence $C_2$ has at most $r+2\ell$ vertices and fits in a card.
We see it as a csc, and its long offshoot from $u_j$ is $T_2$.

\smallskip
{\bf Case 2:} {\it $Q$ has in total at most $\ell$ vertices.}
As remarked earlier, we see $Q$ in full in $C^*$.
Altogether there are at least $3\ell+1$ vertices outside $P$ in $T$.
Since $C^*$ is a card, there are at most $\ell$ vertices outside $P$ in $T_2$,
and $Q$ has at most $\ell$ vertices.  Hence $\C{V(T_1)}\ge j+\ell$.

Let $C_3$ be a largest subtree having an $r$-vertex path $P'$ with vertices
$\la \VEC u1r\ra$ such that $u_j$ has degree $2$ and the offshoots from 
$\{\VEC u1{j-1}\}$ have a total of $\ell+1$ vertices.  Since
$\C{V(T_1)}\ge j+\ell$, there exists such a subtree with $u_j=z$.
Since $T_2$ has at most $\ell$ vertices outside $P$, the offshoots from
$\{\VEC u1{j-1}\}$ guarantee that $u_j$ is not in $T_2$.  Now its position
along $P'$ also prevents $u_j$ from being in $T_1$ or $Q$ (since otherwise
there is too long a path).  Hence $u_j=z$.  Now $C_3$ has at most $r+2\ell+1$
vertices, so it fits in a card, and we see a largest such subtree.  It shows us
$T_2$ in full.
\end{proof}

\end{document}